\documentclass[a4paper, reqno]{amsart}
\usepackage[utf8]{inputenc}

\usepackage{amsthm, amsmath, amssymb}
\usepackage{xcolor}
\usepackage{enumitem}
\usepackage{graphicx}
\usepackage{pgfplots}
\usepackage{hyperref}

\newtheorem{theorem}{Theorem}
\newtheorem{corollary}[theorem]{Corollary}
\newtheorem{lemma}[theorem]{Lemma}

\theoremstyle{definition}
\newtheorem{example}[theorem]{Example}
\newtheorem{definition}[theorem]{Definition}
\newtheorem{remark}[theorem]{Remark}

\newcommand{\qedMP}{}

\newcommand{\N}{\mathbb{N}}
\newcommand{\Z}{\mathbb{Z}}
\newcommand{\Q}{\mathbb{Q}}
\newcommand{\R}{\mathbb{R}}
\newcommand{\C}{\mathbb{C}}

\newcommand{\func}{f}

\newcommand{\funcmin}{\func_{\min}}

\newcommand{\matspace}{\mathcal{S}^k}

\newcommand{\redpolys}[1]{\mathcal{R}[{#1}]}

\newcommand{\kraw}{\mathcal{K}^n}
\newcommand{\krawnorm}{\widehat{\mathcal{K}}^n}
\newcommand{\krawnormnosup}{\widehat{\mathcal{K}}}

\newcommand{\krawlim}{\widehat{\mathcal{K}}^\infty}

\newcommand{\charsum}{X}
\newcommand{\cube}[1]{\mathbb{B}^{#1}}
\newcommand{\qcube}[1]{\mathbb (\Z / q\Z)^{#1}}
\newcommand{\harmbound}{\gamma}

\newcommand{\kernel}{K}
\newcommand{\kernelop}{\mathbf{K}}

\newcommand{\linlamb}{\tilde{\Lambda}}
\newcommand{\nonlinlamb}{\Lambda}

\newcommand{\upbound}[1]{\smash{\func^{({#1})}}}
\newcommand{\upboundsym}[1]{\smash{\func_{\rm{sym}}^{({#1})}}}
\newcommand{\lowbound}[1]{\smash{\func_{({#1})}}}

\DeclareMathOperator{\Aut}{Aut}
\newcommand{\stab}{{\rm{St}}}
\DeclareMathOperator{\spann}{span}
\newcommand{\Sym}{\text{Sym}}

\title{Sum-of-squares hierarchies for binary polynomial optimization}
\author{Monique Laurent}
\address{M. Laurent, Centrum Wiskunde \& Informatica (CWI), Amsterdam and Tilburg University}
\email{monique.laurent@cwi.nl}
\author{Lucas Slot}
\address{L. Slot, Centrum Wiskunde \& Informatica (CWI), Amsterdam}
\email{lucas.slot@cwi.nl}
\thanks{%
This work is supported by the European Union's Framework Programme for Research and Innovation Horizon
2020 under the Marie Skłodowska-Curie Actions Grant Agreement No. 764759 (MINOA)}
\date{\today}

\begin{document}

\begin{abstract}
\sloppy
We consider the sum-of-squares hierarchy of approximations for the problem of minimizing a polynomial $f$ over the boolean hypercube ${\cube{n}=\{0,1\}^n}$. This hierarchy provides for each integer $r \in \N$ a lower bound $\lowbound{r}$ on the minimum $\funcmin$ of $f$, given by the largest scalar $\lambda$ for which the polynomial $f - \lambda$ is a sum-of-squares on $\cube{n}$ with degree at most $2r$. 
We analyze the quality of these bounds by estimating the worst-case error $\funcmin - \lowbound{r}$ in terms of the least roots of the Krawtchouk polynomials. As a consequence, 
for fixed  $t \in [0, 1/2]$, we can show that  this worst-case error  in the regime $r \approx t \cdot n$  is of the order $1/2 - \sqrt{t(1-t)}$ as $n$ tends to $\infty$. 
Our proof 
combines   classical Fourier analysis on $\cube{n}$ with the 
polynomial kernel technique and existing results on the extremal roots of Krawtchouk polynomials.
This link to roots of orthogonal polynomials relies on a  connection between the hierarchy of lower bounds $\lowbound{r}$ and another hierarchy of upper bounds $\upbound{r}$, for which we are also able to establish the same error analysis. Our analysis extends to the minimization of a polynomial over the $q$-ary cube $\qcube{n}$. Furthermore, our results apply to the setting of matrix-valued polynomials.
\end{abstract}

\maketitle


\section{Introduction}
We consider the problem of minimizing a polynomial $f \in \R[x]$ of degree $d \leq n$ 
over the $n$-dimensional boolean hypercube $\cube{n} = \{0, 1\}^n$, i.e., of computing
\begin{equation}
\label{EQ:mainprob}
        \funcmin := \min_{x \in \cube{n}} f(x).
\end{equation}
This optimization problem is NP-hard in general. Indeed, as is well-known, one can model an instance of \textsc{max-cut} on the complete graph $K_n$ with edge weights $w=(w_{ij})$ as a problem of the form \eqref{EQ:mainprob} by setting:
\[
	f(x) = -\sum_{1\le i < j\le n} w_{ij} (x_i - x_j)^2. 
\]
As another example one can compute the stability number $\alpha(G)$ of a graph $G=(V,E)$ via the program
$$\alpha(G) = \max_{x\in\cube{|V|}} \sum_{i\in V}x_i -\sum_{\{i,j\}\in E} x_ix_j.$$
{One may replace  the boolean cube $\cube{n}=\{0,1\}^n$ by   the discrete cube $\{\pm 1\}^n$, in which case maximizing a quadratic polynomial $x^\top Ax$  has many other applications, e.g., to \textsc{max-cut} \cite{GW}, to the cut norm \cite{AN2004}, or to correlation clustering \cite{BBC}. Approximation algorithms are known depending on the structure of the matrix $A$ (see \cite{AN2004,CW2004,GW}), but the problem is known to be NP-hard to approximate within any  factor less than 13/11 \cite{ABHKS}.
}

Problem \eqref{EQ:mainprob} also permits to capture polynomial optimization over  a general region of the form $\cube{n}\cap P$ where $P$ is a polyhedron  \cite{LasserreORL2016} and thus a broad range of combinatorial optimization problems.
The general intractability of problem~\eqref{EQ:mainprob}  motivates the search for tractable bounds on the minimum value in  (\ref{EQ:mainprob}).
 In particular, several lift-and-project methods have been proposed, based on lifting the problem to higher dimension by introducing new variables modelling higher degree monomials. Such methods also apply to constrained problems on $\cube{n}$ where the constraints can be linear or polynomial; see, e.g.,   \cite{BCC1993}, \cite{Lasserre2001b}, \cite{LovaszSchrijver1991}, \cite{Rothvoss}, \cite{SheraliAdams1990}, \cite{Tuncel2010}.
In  \cite{Laurent2001} it is shown that the sums-of-squares hierarchy of Lasserre \cite{Lasserre2001b} in fact refines the other proposed hierarchies.  
As a consequence the sum-of-squares approach for polynomial optimization over $\cube{n}$ has received a great deal of attention in the recent years and there is a vast literature on this topic.
Among many other results, let us just mention its use to show lower bounds on the size of semidefinite programming relaxations for combinatorial problems such as max-cut, maximum stable sets  and TSP in \cite{Lee}, and the links to the Unique Game Conjecture in \cite{Barak}. 
For background about the sum-of-squares hierarchy applied to polynomial optimization over general semialgebraic sets we refer to \cite{Lasserre2001}, \cite{Lasserre2009}, \cite{Laurent2009}, \cite{Parrilo2000} and further references therein.


This  motivates the interest in gaining a better understanding of  the quality of the bounds produced by the sum-of-squares hierarchy. Our objective in this paper is to investigate such an error analysis for this hierarchy applied to binary polynomial optimization as in   \eqref{EQ:mainprob}.

\subsection{The sum-of-squares hierarchy on the boolean cube}
The \emph{sum-of-squares hierarchy} was  introduced by Lasserre \cite{Lasserre2001,Lasserre2001b} and Parrilo \cite{Parrilo2000} as a tool to produce tractable lower bounds for polynomial optimization problems. When applied to problem \eqref{EQ:mainprob} it  provides  for any integer $r\in \N$ a lower bound $\lowbound{r} \leq \funcmin$ on $\funcmin$, given by:
\begin{equation}
	\label{EQ:lowbound}
    \lowbound{r} := \sup_{\lambda \in \R} \left\{f(x) - \lambda \text{ is a sum-of-squares of degree at most } 2r \text{ on } \cube{n} \right\}.
\end{equation}
The condition `$f(x)-\lambda$ is a sum-of-squares of degree at most $2r$ on $\cube{n}$' means that 
 there exists a sum-of-squares polynomial $s \in \Sigma_r$ such that $f(x) - \lambda = s(x)$ for all $x \in \cube{n}$, or, equivalently, the polynomial $f-\lambda-s$ belongs to the ideal generated by the polynomials $x_1-x_1^2,\ldots,x_n-x_n^2$.
Throughout, $\Sigma_r$ denotes the set of sum-of-squares polynomials with degree at most $2r$, i.e., of the form $\sum_i p_i^2$  with $p_i\in \R[x]_r$.
 
{ 
As sums of squares of polynomials can be modelled using semidefinite programming,  problem  \eqref{EQ:lowbound} can be reformulated as a semidefinite program of size polynomial in $n$ for fixed $r$ \cite{Lasserre2001,Parrilo2000}, see also \cite{Laurentfinite}. In the case of unconstrained boolean optimization, the resulting semidefinite program is known to have an optimum solution with small coefficients (see \cite{ODonnell2017} and \cite{Raghavendra2017}). For fixed $r$, the parameter $\lowbound{r}$ may therefore be computed efficiently (up to any precision). 
}

The bounds $\lowbound{r}$ have finite convergence: $\lowbound{r}=f_{\min}$ for $r\ge n$ \cite{Lasserre2001b,Laurent2001}.  In fact, it 
has been shown in 
\cite{SakaueTakedaKim2017} 
that the bound $\lowbound{r}$ is exact already for $2r \geq n+d-1$. That is, 
\begin{equation}
    \label{EQ:exactbound}
    \lowbound{r} = \funcmin \text{ for } r \geq \frac{n+d-1}{2}.
\end{equation}
{In addition, it is shown in  \cite{SakaueTakedaKim2017} that  the bound $\lowbound{r}$ is exact for  ${2r\ge n+d-2}$ when the polynomial $f$ has only monomials of even degree. This  extends an earlier result of  \cite{FawziSaundersonParillo2016} shown for quadratic forms ($d=2$), which applies in particular to the case of \textsc{max-cut}. Furthermore, this result is tight for \textsc{max-cut}, since  one needs to go up to order $2r\ge n$ in order to reach finite convergence (in the cardinality case when all edge weights are 1) \cite{Laurent2003}. Similarly, the result \eqref{EQ:exactbound} is tight when $d$ is even and $n$ is odd \cite{KurpiszLeppanenMastrolilli}.

The main contribution of this work is an analysis of the quality of the bounds $\lowbound{r}$ for parameters $r, n \in \N$ which fall outside of this regime, i.e., $2r<n+d-1$. 
The following is our main result, which expresses the error of the bound $\lowbound{r}$ in terms of the roots of Krawtchouk polynomials, which are classical univariate orthogonal polynomials with respect to a discrete measure on the set $\{0, 1, \ldots, n\}$ (see Section \ref{SEC:Fourier} for details).

\begin{theorem}
\label{THM:main}
Fix $d \leq n$ and let $f \in \R[x]$ be a polynomial of degree $d$. For $r, n \in \N$, let $\xi^n_{r}$ be the least root of the degree $r$ Krawtchouk polynomial \eqref{EQ:krawdef} with parameter $n$.
Then, if  $(r+1)/n \leq 1/2$ and ${d(d+1) \cdot \xi_{r+1}^n / n \leq 1/2}$, we have:
\begin{align}\label{eqfrup}
	\frac{\funcmin - \lowbound{r}}{\|f\|_{\infty}} \leq 2 C_d \cdot \xi_{r+1}^n / n.
\end{align}
Here $C_d > 0$ is an absolute  constant depending only on $d$ and we set 
$\|f\|_\infty := \max_{x \in \cube{n}} | f(x) |.$
\end{theorem}

The extremal roots of Krawtchouk polynomials are well-studied in the literature. The following result of Levenshtein \cite{Levenshtein}
shows their asymptotic behaviour. 
\begin{theorem}[\cite{Levenshtein}, Section 5]
\label{THM:Levenshtein}
For $t \in [0, 1/2]$, define the function
\begin{equation}\label{eqphitt}
	\varphi(t) = 1/2 - \sqrt{t(1-t)}.
\end{equation}
Then the least root $\xi_r^n$ of the degree $r$ Krawtchouk polynomial with parameter $n$ satisfies
\begin{equation}
\label{EQ:Levenshtein}
	\xi_r^n / n \leq \varphi(r/n) + c \cdot (r / n)^{-1/6} \cdot n^{-2/3}
\end{equation}
for some universal constant $c > 0$.
\end{theorem}

Applying \eqref{EQ:Levenshtein} to \eqref{eqfrup}, we find that the relative error of the bound $\lowbound{r}$ in the regime $r \approx t \cdot n$ behaves as the function $\varphi(t) = 1/2 - \sqrt{t(1-t)}$, up to a term in $O(1/n^{2/3})$, which vanishes as $n$ tends to $\infty$.  As an illustration, Figure~\ref{FIG:qphi} in Appendix \ref{APP:qary} shows the function $\varphi(t)$.

\subsection{A second hierarchy of bounds}
In addition to the \emph{lower} bound $\lowbound{r}$, Lasserre \cite{Lasserre2010} also defines an \emph{upper} bound $\upbound{r} \geq \funcmin$ on $\funcmin$ as follows:
\begin{equation}
\label{EQ:upbound}
	\upbound{r} := \inf_{s \in \Sigma_r} \left \{ \int_{\cube{n}} f(x) \cdot s(x) d\mu(x) : \int_{\cube{n}} s(x) d\mu(x) = 1\right \},
\end{equation}
where $\mu$ is the uniform probability measure on $\cube{n}$.
%
For fixed $r$, similarly to $\lowbound{r}$, one may compute $\upbound{r}$ (up to any precision) efficiently  by reformulating  problem \eqref{EQ:upbound} as a semidefinite program \cite{Lasserre2010}. Furthermore, as  shown in \cite{Lasserre2010} the bound is exact for some order $r$,  and it is not difficult to see that the bound $\upbound{r}$ is exact at order $r = n$ and that this is tight (see Section \ref{SEC:generalize}).

Essentially as a side result in the proof of Theorem \ref{THM:main}, we can show the following analog for the upper bounds $\upbound{r}$, which we believe to be of independent interest.

\begin{theorem}
\label{THM:mainup}
Fix $d \leq n$ and let $f \in \R[x]$ be a polynomial of degree $d$. 
Then, for any $r, n \in \N$ with $(r+1)/n \leq 1/2$, we have:
\begin{align*}
	\frac{\upbound{r} - \funcmin }{\|f\|_{\infty}} \leq C_d \cdot \xi_{r+1}^n / n,
\end{align*}
where $C_d > 0$ is the constant mentioned in Theorem \ref{THM:main}.
\end{theorem}
So we have  the same estimate of the relative error for the upper bounds $\upbound{r}$ as for the lower bounds $\lowbound{r}$ (up to a constant factor $2$) and indeed we will see that our proof relies on an intimate connection between both hierarchies.
Note that the above analysis of $\upbound{r}$ does not require any condition on the size of $\xi_{r+1}^n$ as was necessary for the analysis of $\lowbound{r}$ in Theorem \ref{THM:main}. Indeed, as will become clear later, the condition put on $\xi_{r+1}^n$ follows from a technical argument (see Lemma \ref{LEM:normbound}), which is not required in the proof of Theorem \ref{THM:mainup}.

\subsection{Asymptotic analysis for both hierarchies}
The results above show that the relative error of both 
hierarchies is bounded  asymptotically by the function $\varphi(t)$ from (\ref{eqphitt}) in the regime $r\approx t\cdot n$. This is summarized in the following corollary, which can be seen as an asymptotic version of Theorem \ref{THM:main} and Theorem \ref{THM:mainup}.

\begin{corollary}\label{cor2}
Fix $d\le n$ and for $n,r\in\N$ write
\begin{align*}
E_{(r)}(n) &:= \sup_{f \in \R[x]_d} \big\{\funcmin - \lowbound{r} : \|f\|_\infty = 1 \big\}, \\ 
E^{(r)}(n) &:= \sup_{f \in \R[x]_d} \big\{ \upbound{r} - \funcmin : \|f\|_\infty = 1 \big\}.
\end{align*}
Let $C_d$ be the constant of Theorem \ref{THM:main} and let  $\varphi(t)$ be the function from (\ref{eqphitt}). Then, for any $t\in [0,1/2]$, we have:
$$
\lim_{r/n\to t} E^{(r)}(n)  \le C_d\cdot \varphi(t)
$$
and, if $d(d+1)\cdot \varphi(t) \le 1/2$, we also have:
$$
\lim_{r/n\to t} E_{(r)}(n)  \le 2\cdot C_d\cdot \varphi(t).
$$
 Here, the  limit notation $r/n\to t$  means that the claimed convergence holds for all sequences $(n_j)_j$ and $(r_j)_j$ of integers such that $\lim_{j\to\infty} n_j=\infty$ and $\lim_{j\to\infty}r_j/n_j=t$. 
\end{corollary}

We close with some remarks. First, note that $\varphi(1/2) = 0$. Hence  Corollary \ref{cor2} tells us that the relative error of both 
hierarchies tends to $0$ as $r / n \to  1/2$. We thus `asymptotically' recover the exactness result \eqref{EQ:exactbound} of \cite{SakaueTakedaKim2017}.

Our results in Theorems \ref{THM:main} and \ref{THM:mainup} and Corollary \ref{cor2} extend directly to the case of polynomial optimization over the discrete cube $\{\pm 1\}^n$ instead of the boolean cube $\cube{n}=\{0,1\}^n$, as can easily be seen by applying a change of variables $x\in \{0,1\}\mapsto 2x-1\in \{\pm 1\}$. 
In addition, as we show in Appendix \ref{APP:qary}, our results extend to the case of polynomial optimization over the $q$-ary cube $\{0,1,\ldots,q-1\}^n$ for $q > 2$.

After replacing $f$ by its negation $-f$, one may use the \emph{lower} bound $\lowbound{r}$ on $\funcmin$ in order to obtain an \emph{upper} bound on the \emph{maximum} $f_{\max}$ of $f$ over $\cube{n}$. Similarly, one may obtain a \emph{lower} bound on $f_{\max}$ using the \emph{upper} bound $\upbound{r}$ on $\funcmin$. To avoid possible confusion, we will also refer to $\lowbound{r}$ as the \emph{outer} Lasserre hierarchy (or simply the sum-of-squares hierarchy), whereas we will refer to $\upbound{r}$ as the \emph{inner} Lasserre hierarchy. This terminology  (borrowed from \cite{deKlerkLaurentSurvey}) is motivated by the following observations.
As is well-known (and easy to see) the parameter $f_{\min}$ can be reformulated as an optimization problem over the set $\mathcal M$ of Borel measures on $\cube{n}$:
$$f_{\min}=\min\Big\{\int_{\cube{n}} f(x)d\nu(x): \nu\in \mathcal M, \int_{\cube{n}} d\nu(x)=1\Big\}.$$
If we replace the set $\mathcal M$ by its {\em inner} approximation consisting of all measures $\nu(x)=s(x)d\mu(x)$ with polynomial density  $s\in\Sigma_r$ with respect to a given fixed measure $\mu$, then we obtain the bound $f^{(r)}$.
On the other hand, any $\nu\in\mathcal M$  corresponds to a  linear functional $L_\nu:p\in \mathbb R[x]_{2r}\mapsto \int_{\cube{n}}p(x)d\nu(x)$ which is nonnegative on sums-of-squares on $\cube{n}$. 
These linear functionals thus provide an {\em outer} approximation for $\mathcal M$ and  maximizing $L_\nu(p)$ over it gives the bound $f_{(r)}$ (in dual formulation).

\subsection{Related work}
As mentioned above, the bounds $\lowbound{r}$ defined in \eqref{EQ:lowbound} are known to be exact when $2r \geq n+d -1$. The case $d=2$ (which includes \textsc{max-cut}) was treated in \cite{FawziSaundersonParillo2016}, positively answering a question posed in \cite{Laurent2003}. Extending the strategy of \cite{FawziSaundersonParillo2016}, the general case was settled in \cite{SakaueTakedaKim2017}. These exactness results are best possible when $d$ is even and $n$ is odd \cite{KurpiszLeppanenMastrolilli}.

In \cite{KarlinMathieuNguyen2011}, the sum-of-squares hierarchy is considered for approximating instances of \textsc{knapsack}. This can be seen as a variation on the problem \eqref{EQ:mainprob}, restricting to a linear polynomial objective with positive coefficients, but introducing a single, linear constraint, of the form $a_1x_1+\ldots+a_nx_n\le b$ with $a_i>0$. There, the authors show that the outer hierarchy has relative error at most $1 / (r - 1)$ for any integer $r\ge 2$. 
To the best of our knowledge this is the only known case where one can analyze the quality of the outer bounds for {\em all} orders $r\le n$.

For optimization over sets other than the boolean cube, the following results on the quality of the outer hierarchy $\lowbound{r}$ are available. When considering general semi-algebraic sets (satisfying a compactness condition), it has been shown in  \cite{NieSchweighofer2007} that there exists a constant $c>0$ (depending on the semi-algebraic set) such that $\lowbound{r}$ converges to $\funcmin$ at a rate in $O(1 / \log(r/c)^{1/c})$ as $r$ tends to $\infty$. This rate can be improved  to $O(1 / r^{1/c})$ if one considers a variation of the sum-of-squares hierarchy which is stronger (based on the preordering instead of the quadratic module), but much more computationally intensive \cite{Schweighofer2004}.  
Specializing to the hypersphere $S^{n-1}$, better rates in $O(1/r)$ were shown in \cite{Reznick1995,Doherty_Wehner2012}, and recently improved to $O(1/r^2)$ in \cite{FangFawzi2019}. Similar improved results exist also for the case of polynomial optimization on the simplex and the continuous hypercube $[-1,1]^n$; we refer, e.g.,  to \cite{deKlerkLaurentSurvey} for an overview.


{
The results for semi-algebraic sets other than $\cube{n}$ mentioned above  all apply in the asymptotic regime where the dimension $n$ is fixed and $r \to \infty$. This makes it difficult to compare them directly to our new results. Indeed, we have to consider a different regime in the case of the boolean cube $\cube{n}$, as the hierarchy always converges in at most $n$ steps. The regime where we are able to provide an analysis in this paper is when $r \approx t\cdot n$ with $0<t\le 1/2$.
}

Turning now to the \emph{inner} hierarchy \eqref{EQ:upbound}, as far as we are aware, nothing is known about the behaviour of the bounds $\upbound{r}$ on $\cube{n}$. For full-dimensional compact sets, however, results are available. It has been shown that, on  the hypersphere \cite{deKlerkLaurent2019}, the unit ball and the simplex \cite{SlotLaurent2019}, and the unit box \cite{deKlerkLaurent2018}, the bound $\upbound{r}$ converges at a rate in $O(1/r^2)$. A slightly weaker convergence rate in $O(\log^2 r / r^2)$ is known for general (full-dimensional) semi-algebraic sets \cite{SlotLaurent2019,SlotLaurent2020}. Again, these results are all asymptotic in $r$, and thus hard to compare directly to our analysis on $\cube{n}$.

\subsection{Overview of the proof}
\label{SEC:overview}

Here, we give an  overview of the main ideas that we use to show our results. Our broad strategy follows the one employed in \cite{FangFawzi2019} to obtain information on the sum-of-squares hierarchy on the hypersphere. The following four ingredients will play a key role in our proof:
\begin{enumerate}
\item we use the {\em polynomial kernel technique} in order to produce low-degree sum-of-squares representations of polynomials that are positive over $\cube{n}$, thus allowing an analysis of $\funcmin - \lowbound{r}$;
\item  using classical \emph{Fourier analysis} on the boolean cube $\cube{n}$ we are able to exploit symmetry and  reduce the search for a \emph{multivariate kernel}  to a \emph{univariate sum-of-squares polynomial} on the discrete set $[0:n] := \{0,1,\ldots,n\}$;
\item we find this univariate sum-of-squares by applying the {\em inner} Lasserre hierarchy to an appropriate univariate optimization problem on $[0:n]$;
\item finally, we exploit a known connection between the inner hierarchy and the \emph{extremal roots of corresponding orthogonal polynomials} (in our case, the Krawtchouk polynomials).
\end{enumerate}
Following these steps we are able to analyze the sum-of-squares hierarchy $\lowbound{r}$  as well as the inner hierarchy  $\upbound{r}$. We  now sketch  how our proof articulates along these four main steps.

\smallskip
Let $f\in \mathbb R[x]_d$ be the polynomial with degree $d$ for which we wish to analyze the bounds $f_{(r)}$ and $f^{(r)}$. After rescaling, and up to a change of coordinates, we may assume w.l.o.g. that $f$ attains its minimum over $\cube{n}$  at $0 \in \cube{n}$ and that $f_{\min}=0$ and $f_{\max} = 1$. So we have $\|f\|_\infty=1$. To simplify notation, we will make these assumptions throughout.

The first key idea is to consider a \emph{polynomial kernel} $\kernel$ on $\cube{n}$ of the form:
\begin{equation}
\label{EQ:kernelform}
\kernel(x, y) = u^2(d(x,y)),
\end{equation}
where $u \in \R[t]_r$ is a univariate polynomial of degree at most $r$ and $d(x, y)$ is the Hamming distance between $x$ and $y$. 
Such a kernel $K$ induces an operator $\kernelop$, which acts linearly on the space of polynomials on $\cube{n}$ by:
\begin{equation*}
p\in \mathbb R[x] \mapsto 	\kernelop p(x) := \int_{\cube{n}} p(y) \kernel(x,y) d\mu(y) =  {1\over 2^n} \sum_{y \in \cube{n}} p(y) \kernel(x, y).
\end{equation*}
Recall that $\mu$ is the uniform probability distribution on $\cube{n}$. An easy but important observation is that, if $p$ is nonnegative on $\cube{n}$, then $\kernelop p$ is a sum-of-squares (on $\cube{n}$) of degree at most $2r$. 
We use this fact  as follows.

Given a scalar $\delta \geq 0$, define the polynomial $\tilde f := f + \delta$. Assuming that the operator $\kernelop$ is non-singular, we can express $\tilde f$ as $\tilde f = \kernelop (\kernelop^{-1} \tilde f)$. Therefore, if $\kernelop^{-1} \tilde f$ is nonnegative on $\cube{n}$, then we find that $\tilde f$ is a sum-of-squares on $\cube{n}$ with degree at most $2r$, which implies   $\funcmin - \lowbound{r} \leq \delta$.

One way to guarantee that $\kernelop^{-1} \tilde f$ is indeed nonnegative on $\cube{n}$  is to select the operator $\kernelop$ in such a way that $\kernelop (1)=1$ and 
\begin{equation}
	\label{EQ:overview1}
\|\kernelop^{-1}-I\| := \sup_{p \in \R[x]_d} \frac{\| \kernelop^{-1} p - p \|_\infty}{\|p\|_\infty} \leq \delta.
\end{equation}
We collect this as a lemma for further reference.

\begin{lemma}\label{lemK0}
If the kernel operator $\kernelop$ associated to $u \in \R[t]_r$ via relation (\ref{EQ:kernelform}) satisfies
$\kernelop (1)=1$ and $\|\kernelop ^{-1}-I\|\le \delta$, then we have $f_{\min}-f_{(r)}\le \delta$.
\end{lemma}

\begin{proof}
With $\tilde f=f+\delta$, we have:
$\|\kernelop^{-1}\tilde f-\tilde f\|_\infty=\|\kernelop ^{-1}f -f\|_\infty \le \delta\|f\|_\infty=\delta$, where we use the fact that $\kernelop (1) = 1 = \kernelop^{-1} (1)$ and thus $\kernelop^{-1}\tilde f-\tilde f =\kernelop ^{-1}f -f$  for the first equality and (\ref{EQ:overview1}) for the inequality.
The inequality $\|\kernelop^{-1}\tilde f-\tilde f\|_\infty\le\delta$ then implies  $\kernelop^{-1}\tilde f(x)\ge \tilde f(x)-\delta=f(x)\ge f_{\min}=0$ on $\cube{n}$. \qedMP
\end{proof}

In other words, we want the operator $\kernelop^{-1}$ (and thus $\kernelop$) to be `close to the identity operator' in a certain sense. 
As kernels of the form \eqref{EQ:kernelform} are invariant under the symmetries of $\cube{n}$, we are able to use classical Fourier analysis on the boolean cube to express the eigenvalues $\lambda_i$ of $\kernelop$ in terms of the polynomial $u$. More precisely, it turns out that these eigenvalues are given by the coefficients of the expansion of $u^2$ in the basis of \emph{Krawtchouk polynomials}. As we show later (see \eqref{EQ:nonlinlamb}), inequality \eqref{EQ:overview1} then holds for $\delta \approx \sum_{i} |\lambda_i^{-1} - 1|$.

It therefore remains to find a univariate polynomial $u \in \R[t]_r$ for which these coefficients, and thus the eigenvalues of $\kernelop$, are sufficiently close to $1$. Interestingly, this problem boils down to analyzing the quality of the inner bound $g^{(r)}$ (see \eqref{EQ:upbound}) for a particular univariate polynomial $g$ (given   in (\ref{EQ:qopt})).

In order to perform this analysis and conclude the proof of Theorem \ref{THM:main}, we make use of a connection between the inner Lasserre hierarchy and the least roots of orthogonal polynomials (in this case the Krawtchouk polynomials).

Finally, we generalize our analysis of the inner hierarchy (for the special case of  the selected polynomial $g$ in Step 3 to arbitrary polynomials)  to obtain Theorem \ref{THM:mainup}.

\subsection{Matrix-valued polynomials}
As we explain in this section the results of Theorem \ref{THM:main} and Theorem \ref{THM:mainup} may be extended to the setting of matrix-valued polynomials.

For fixed $k \in \N$, let $\matspace \subseteq \R^{k \times k}$ be the space of $k \times k$ real symmetric matrices.  We write $\matspace[x] \subseteq \R^{k \times k}[x]$ for the space of $n$-variate polynomials whose coefficients lie in $\matspace$, i.e., the set of $k\times k$ symmetric polynomial matrices.  The polynomial optimization problem \eqref{EQ:mainprob} may be generalized to polynomials $F \in \matspace[x]$ as:
\begin{equation}
	F_{\min} := \min_{x \in \cube{n}} \lambda_{\min}(F(x)),
\end{equation}
where $F\in \matspace[x]$ is a polynomial matrix and $\lambda_{\min}(F(x))$ denotes the smallest eigenvalue of $F(x)$. That is, $F_{\min}$ is the largest value for which $F(x) - F_{\min} I~{ \succeq 0}$ for all $x \in \cube{n}$, where $\succeq$ is the  positive semidefinite L\"owner order.

The Lasserre hierarchies \eqref{EQ:lowbound} and \eqref{EQ:upbound} may also be defined in this setting. A polynomial matrix $S \in \matspace[x]$ is called a sum-of-squares polynomial  of degree $2r$ if it can be written as:
\[
	S(x) = \sum_{i} U_i(x) U_i(x)^\top \quad (U_i \in \R^{k \times m}[x], ~~{\rm \deg} ~U_i \leq r, ~~ m\in\N).
\]
We write $\Sigma^{k \times k}_r$ for the set of all such polynomials. See, e.g., \cite{SchererHol} for background and results on sum-of-squares polynomial matrices.
Note that a sum-of-squares polynomial matrix $S \in \Sigma^{k \times k}_r$ satisfies $S(x) \succeq 0$ for all $x \in \cube{n}$. 
For a matrix $A \in \matspace$,  $\|A\|$ denotes its spectral (or operator) norm, defined as the largest absolute value of an eigenvalue of $A$. Then, for $F\in \matspace[x]$, we set 
$$\|F\|_\infty = \max_{x\in \cube{n}} \|F(x)\|.$$

The outer hierarchy $F_{(r)}$ for the polynomial $F$ is then given by:
\begin{equation}
    F_{(r)} := \sup_{\lambda \in \R} \left\{F(x) - \lambda \cdot I = S(x) \text{ on } \cube{n} \text{ for some } S \in \Sigma^{k \times k}_r \right\}
\end{equation}
Similarly, the inner hierarchy $F^{(r)}$ is defined as:
\begin{equation}
	 F^{(r)} := \inf_{S \in \Sigma^{k \times k}_r} \left \{ \int_{\cube{n}} {\rm Tr} \big( F(x) S(x)\big) d\mu(x) : \int_{\cube{n}} {\rm Tr} \big(S(x)\big) d\mu(x) = 1\right \}.
\end{equation}
Indeed, \eqref{EQ:lowbound} and \eqref{EQ:upbound} are the special cases of the above programs when $k=1$.
As before, the parameters $F_{(r)}$ and $F^{(r)}$ may be computed efficiently for fixed $r$ (and fixed $k$) using semidefinite programming, and they satisfy \[F_{(r)} \leq F_{(r+1)} \leq F_{\min} \leq F^{(r+1)} \leq F^{(r)}.\]

As was already noted in \cite{FangFawzi2019} (in the context of optimization over the unit sphere), the proof strategy outlined in Section~\ref{SEC:overview} naturally extends to the setting of  polynomial matrices. This yields the following generalizations of our main theorems.
\begin{theorem}
\label{THM:mainmatrix}
\sloppy
Fix $d \leq n$ and let $F \in \matspace[x]$ ($k\ge 1$) be a  polynomial matrix of degree $d$. For $r, n \in \N$, let $\xi^n_{r}$ be the least root of the degree $r$ Krawtchouk polynomial \eqref{EQ:krawdef} with parameter $n$.
Then, if $(r+1)/n \leq 1/2$ and ${d(d+1) \cdot \xi_{r+1}^n / n \leq 1/2}$, we have:
\begin{align}
	\frac{F_{\min} - F_{(r)} }{\|F\|_{\infty}} \leq 2 C_d \cdot \xi_{r+1}^n / n,
\end{align}
where  $C_d > 0$ is the absolute  constant depending only on $d$ from  Theorem \ref{THM:main}.
\end{theorem}
\begin{theorem}
\label{THM:mainupmatrix}
Fix $d \leq n$ and let $F \in \matspace[x]$ be a matrix-valued polynomial of degree $d$.
Then, for any $r, n \in \N$ with $(r+1)/n \leq 1/2$, we have:
\begin{align*}
	\frac{F^{(r)} - F_{\min} }{\|F\|_{\infty}} \leq C_d \cdot \xi_{r+1}^n / n,
\end{align*}
where $C_d > 0$ is the constant of Theorem \ref{THM:main}.
\end{theorem}

\subsection*{Organization}
The rest of the paper is structured as follows. We review the necessary background on Fourier analysis on the boolean cube in Section \ref{SEC:Fourier}. In Section \ref{SEC:innerbound}, we recall a connection between the inner Lasserre hierarchy and the roots of orthogonal polynomials. Then, in Section \ref{SEC:mainproof}, we give a proof of  Theorem \ref{THM:main}. In Section \ref{SEC:generalize}, we discuss how to generalize the proofs of Section \ref{SEC:mainproof} to obtain Theorem \ref{THM:mainup}.
We group the proofs of some  technical results needed to prove Lemma~\ref{LEM:cube}  in Section \ref{APP:harmcompbound}. In Appendix \ref{APP:qary}, we indicate how our arguments extend to the case of polynomial optimization over the $q$-ary hypercube $\{0,1,\ldots,q-1\}^n$ for $q>2$. Finally, we give the proofs of our results in the matrix-valued setting (Theorem \ref{THM:mainmatrix} and Theorem \ref{THM:mainupmatrix}) in Appendix \ref{APP:matrixvalued}.

\section{Fourier analysis on the boolean cube}
\label{SEC:Fourier}
In this section, we cover some standard Fourier analysis on the boolean cube. For a general reference on Fourier analysis on (finite, abelian) groups, see e.g. \cite{Terras1999}. See also Section~4 of \cite{VallentinLN} for the boolean case.

\subsection*{Some notation}
For $n \in \N$, we write $\cube{n} = \{0, 1\}^n$ for the boolean hypercube of dimension $n$.
We let $\mu$ denote  the uniform probability measure on $\cube{n}$, given by $\mu = \frac{1}{2^n} \sum_{x \in \cube{n}}\delta_x$, where $\delta_x$ is the Dirac measure at $x$. Further, we write $|x| = \sum_{i}x_i = |\{i\in [n]: x_i = 1 \}|$ for the \emph{Hamming weight} of $x \in \cube{n}$, and $d(x,y) = |\{i\in [n]: x_i \neq y_i\}|$ for the \emph{Hamming distance} between $x, y \in \cube{n}$. We let $\Sym(n)$ denote the set of permutations of the set $[n]=\{1,\ldots,n\}$.

We consider polynomials $p : \cube{n} \rightarrow \R$ on $\cube{n}$. The space $\redpolys{x}$ of such polynomials is given by the quotient ring of $\R[x]$ over the equivalence relation $p \sim q$ if $ p(x) = q(x)$ on $\cube{n}$. In other words, $\redpolys{x}= \mathbb R[x]/\mathcal I$, where $\mathcal I$ is the ideal generated by the polynomials $x_i-x_i^2$ for $i\in [n]$, which can also be seen as the vector space spanned by the (multilinear) polynomials $\prod_{i\in I}x_i$ for $I\subseteq [n].$

For $a\le  b \in \N$, we let $[a:b]$ denote the set of integers $a, a + 1, \dots, b$.

\subsection*{The character basis}
Let $\langle \cdot, \cdot \rangle_\mu$ %
be the inner product on $\redpolys{x}$ given by:
\[
    \langle p, q \rangle_{\mu} = \int_{\cube{n}} p(x) q(x) d \mu(x) = \frac{1}{2^n} \sum_{x \in \cube{n}} p(x) q(x).
\]
The space $\redpolys{x}$ has an orthonormal basis w.r.t. $\langle \cdot, \cdot \rangle_{\mu}$ given by the \emph{characters}:
\begin{equation}
	\label{EQ:chardef}
    \chi_a(x) := (-1)^{x \cdot a} = \prod_{i:a_i=1}(1-2x_i) \quad \left(a \in \cube{n}\right).
\end{equation}
In other words, the set $\{ \chi_a : a \in \cube{n} \}$ of all characters on $\cube{n}$ forms a basis for $\redpolys{x}$ and 
\begin{equation}
\label{EQ:characterinner}
    \langle \chi_a, \chi_b \rangle_{\mu} = \frac{1}{2^n} \sum_{x \in \cube{n}} \chi_a(x) \chi_b(x) = \delta_{a, b} \quad \forall a, b \in \cube{n}.
\end{equation}
Then any polynomial $p \in \redpolys{x}$ can be expressed in the basis of characters, known as its   \emph{Fourier expansion}:
\begin{equation}
   \label{EQ:Fourierexp}
   p(x) = \sum_{a \in \cube{n}} \widehat{p}(a) \chi_a(x) \quad \forall x \in \cube{n}
\end{equation}
with \emph{Fourier coefficients} $\widehat{p}(a) := \langle p, \chi_a \rangle_{\mu} \in \R$.

The group $\Aut(\cube{n})$ of automorphisms of $\cube{n}$ is generated by the coordinate permutations, of the form   $x\mapsto \sigma(x):=(x_{\sigma(1)},\ldots,x_{\sigma(n)})$ for $\sigma \in \Sym(n)$, 
and the automorphisms corresponding to bit-flips, of the form $x\in \cube{n}\mapsto x\oplus a\in \cube{n}$ for $a\in\cube{n}$.
If we set
\[
    H_k := \spann \{ \chi_a : |a| = k\} \quad (0 \leq k \leq n),
\]
then each $H_k$ is an irreducible, $\Aut(\cube{n})$-invariant subspace of  $\redpolys{x}$ of dimension ${n \choose k}$. Using \eqref{EQ:Fourierexp}, we may then decompose $\redpolys{x}$ as the direct sum
\[
\redpolys{x} = H_0 \perp H_1 \perp \dots \perp H_n,
\]
where the subspaces $H_k$ are pairwise orthogonal w.r.t. $\langle \cdot, \cdot\rangle_\mu$. In fact, we have that $\redpolys{x}_d = H_0 \perp H_1 \perp \dots \perp H_d$ for all $d \leq n$, and we may thus write any $p \in \redpolys{x}_d$ (in a unique way) as 
\begin{equation}\label{eqdecp}
p = p_0 + p_1 + \dots + p_d \quad (p_k \in H_k).
\end{equation}
The polynomials $p_k\in H_k$ ($k=0,\ldots,d$) are known as the {\em harmonic components} of $p$ and the decomposition (\ref{eqdecp}) as the {\em harmonic decomposition} of $p$. We will make extensive use of this decomposition throughout.

Let $\stab(0) \subseteq \Aut(\cube{n})$ be the set of automorphisms fixing $0 \in \cube{n}$, which consists of the coordinate permutations $x\mapsto \sigma(x) =(x_{\sigma(1)},\ldots,x_{\sigma(n)})$ for $\sigma\in \Sym(n)$. The subspace of functions in $H_k$ that are  invariant under $\stab(0)$ is one-dimensional and it is spanned by the function
\begin{equation}
    \label{EQ:charsum}
        \charsum_k(x) := \sum_{|a| = k} \chi_a(x).
\end{equation}
These functions $\charsum_k$ are known as the \emph{zonal spherical functions} with pole ${0 \in \cube{n}}$.

\subsection*{Krawtchouk polynomials}
For $k \in \N$, the \emph{Krawtchouk polynomial}  of degree $k$ (and with parameter $n$) is the univariate polynomial in $t$ given by:
\begin{equation}
\label{EQ:krawdef}
    \kraw_k(t) := \sum_{i=0}^k (-1)^i {t \choose i} {n-t \choose k-i}
\end{equation}
(see, e.g. \cite{Szego1959}). %
The Krawtchouk polynomials form an orthogonal basis for $\R[t]$ with respect to the inner product given by the following discrete probability measure on the set $[0:n]=\{0,1,\ldots,n\}$:
\begin{equation*}
    \omega := \frac{1}{2^n} \sum_{t=0}^n w(t) \delta_t, \text{ with } w(t) := {n \choose t}.
\end{equation*}
Indeed, for all $k, k' \in \N$ we have:
\begin{equation}
\label{EQ:krawortho}
    \langle \kraw_k, \kraw_{k'}\rangle_\omega :={1\over 2^n} \sum_{t=0}^n \kraw_k(t) \kraw_{k'}(t) w(t) = \delta_{k, k'} {n \choose k} .
\end{equation}
The following (well-known) lemma explains the connection between the Krawtchouk polynomials and the character basis on $\redpolys{x}$.  

\begin{lemma}\label{lemxt}
Let $t \in [0:n]$ and choose $x, y \in \cube{n}$ so that $d(x, y) = t$. Then for any $0\le k \leq n$ we have:
\begin{equation}
\label{EQ:krawcharacter}
        \kraw_k(t) = \sum_{|a| = k} \chi_a(x) \chi_a(y).
\end{equation}
In particular, we have:
\begin{equation}\label{EQ:xt}
    \kraw_k(t) = \sum_{|a| = k} \chi_a(1^t 0^{n-t}) = \charsum_k(1^t 0^{n-t}),
\end{equation}
where $1^t 0^{n-t} \in \cube{n}$ is given by $(1^t 0^{n-t})_i = 1$ if $1\le i \leq t$ and $(1^t 0^{n-t})_i = 0$ if $t+1\le i\le n$.
\end{lemma}
\begin{proof}
Noting that $\chi_a(x) \chi_a(y) = \chi_a(x+y)$ and  $|x + y| = d(x,y) = t$, we have:
\begin{align*}
	\sum_{|a| = k} \chi_a(x) \chi_a(y) &= \sum_{i = 0}^k (-1)^i \cdot \# \{|a| = k : a \cdot (x+y) = i\} \\
	&= \sum_{i = 0}^k (-1)^i {t \choose i} {n - t \choose k - i} = \kraw_k(t).
\end{align*} \qedMP
\end{proof}
From this, we see that any polynomial $p\in\R[x]_d$ that is invariant under the action of $\stab(0)$ is of the form $\sum_{i=1}^d \lambda_i \kraw_i(|x|)$ for some scalars $\lambda_i$, and thus $p(x)=u(|x|)$ for some univariate polynomial $u \in \R[t]_d$.

It will sometimes be convenient to work with a different normalization of the Krawtchouk polynomials, given by:
\begin{equation}
	\label{EQ:krawnorm}
    \krawnorm_k(t) := \kraw_k(t) / \kraw_k(0) \quad (k \in \N).
\end{equation}
So $\krawnorm_k(0)=1$. Note that for any $k \in \N$, we have \[ \|\kraw_k\|^2_\omega := \langle \kraw_k, \kraw_{k}\rangle_\omega = {n \choose k} = \kraw_k(0),\] meaning that $\krawnorm_k(t) = \kraw_k(t) / \|\kraw_k\|^2_\omega$.

Finally we give a short proof of two basic facts on Krawtchouk polynomials that will be used below.

\begin{lemma}
\label{LEM:kraw1}
We have:
\[
\krawnorm_k(t) \leq \krawnorm_0(t) = 1
\]
for all $0\le k \leq n$ and $t \in [0 : n]$.
\end{lemma}

\begin{proof}
Given $t \in [0:n]$ consider an element  $x \in \cube{n}$ with Hamming weight $t$, for instance the element $1^t 0^{n-t}$ from Lemma \ref{lemxt}. 
By \eqref{EQ:xt} 
we have
\[
\kraw_k(t) = 
 \sum_{|a| = k} \chi_a(x) \leq {n \choose k} = \kraw_k(0),
\]
making use of the fact that 
$|\chi_a(x)| = 1$ for all $a \in \cube{n}$. \qedMP
\end{proof}

\begin{lemma}
\label{LEM:krawabsdist}
We have:
\begin{equation}
\label{EQ:krawabsdist}
\begin{aligned}
    |\krawnorm_k(t) - \krawnorm_k(t+1)| &\leq \frac{2k}{n}, \quad\quad &&(t=0,1,\dots,n-1)\\
    |\krawnorm_k(t) - 1| &\le {2k\over n} \cdot t &&(t=0,1,\dots,n)
\end{aligned}
\end{equation}
for all $0\le k \leq n$. 
\end{lemma}

\begin{proof}
Let $t\in [0:n-1]$ and $0 < k \leq d$. Consider the elements $1^t 0^{n-t}$ and $1^{t+1} 0^{n-t-1}$ of $\cube{n}$ from Lemma \ref{lemxt}. We have:
\begin{align*}
|\kraw_k(t) - \kraw_k(t+1)| 
&\overset{\eqref{EQ:xt}}{=}
 |\sum_{|a| = k} \chi_a(1^t 0^{n-t}) - \chi_a(1^{t+1} 0^{n-t-1})| \\
&\leq 2 \cdot \#\big\{ a \in \cube{n} : |a| = k, a_{t+1} = 1\big\} = 2{n - 1 \choose k-1},
\end{align*}
where for the  inequality we note that $\chi_a(1^t 0^{n-t}) = \chi_a(1^{t+1} 0^{n-t-1})$ if $a_{t+1} = 0$. As $\kraw_k(0) = {n \choose k}$, this implies that:
\[
|\krawnorm_k(t) - \krawnorm_k(t+1)| \leq 2{n - 1 \choose k-1} / {n \choose k} = \frac{2k}{n}.
\]
This shows the first inequality of (\ref{EQ:krawabsdist}). The second inequality follows using the triangle inequality, a telescope summation argument and the fact that $\krawnorm_k(0)=1$. \qedMP
\end{proof}

\subsection*{Invariant kernels and the Funk-Hecke formula}
Given a univariate polynomial $u \in \R[t]$ of degree $r$ with $2r \geq d$, consider the kernel ${\kernel : \cube{n} \times \cube{n} \to \R}$ defined by 
\begin{equation}
\label{EQ:kernel}
    \kernel(x, y) := u^2(d(x,y)).
\end{equation}
A kernel of the form \eqref{EQ:kernel} coincides with a polynomial of degree $2\mathrm{deg}(u)$ in $x$ on the binary cube $\cube{n}$, as $d(x,y) = \sum_i (x_i + y_i - 2x_iy_i)$ for $x, y \in \cube{n}$. 
Furthermore, it is invariant under $\Aut(\cube{n})$, in the sense that:
\[
\kernel(x, y) = \kernel(\pi(x), \pi(y)) \quad\forall x,y\in\cube{n}, \pi \in \Aut(\cube{n}).
\]
The kernel $\kernel$ acts as a linear operator $\kernelop : \redpolys{x} \to \redpolys{x}$ by:
\begin{equation}\label{eqKp}
\kernelop p(x) := \int_{\cube{n}} p(y) \kernel(x, y) d\mu(y) = \frac{1}{2^n} \sum_{y \in \cube{n}} p(y) \kernel(x,y).
\end{equation}
We may expand the univariate polynomial $u^2 \in \R[t]_{2r}$ in the basis of Krawtchouk polynomials as:
\begin{equation}
\label{EQ:q2expansion}
        u^2(t) = \sum_{i = 0}^{2r} \lambda_i \kraw_i(t) \quad (\lambda_i \in \R).
\end{equation}
As we show now, the eigenvalues of the operator $\kernelop$ are given precisely by the coefficients $\lambda_i$ occurring in  this expansion. This relation is analogous to the classical Funk-Hecke formula for spherical harmonics (see, e.g., \cite{FangFawzi2019}).

\begin{theorem}[Funk-Hecke]
\label{THM:FunckHecke}
\sloppy
Let $p  \in \redpolys{x}_d$ with harmonic decomposition ${p= p_0 + p_1 + \dots + p_d}$. Then we have:
\begin{equation}
\label{EQ:FunkHecke}
\kernelop p = \lambda_0 p_0 + \lambda_1 p_1 + \dots + \lambda_d p_d.
\end{equation}
\end{theorem}

\begin{proof}
It suffices to show that 
$
\kernelop \chi_z = \lambda_{|z|} \chi_z
$
for all $z \in \cube{n}$.
So we compute for $x \in \cube{n}$:
\begin{align*}
    \kernelop \chi_z(x) 
    = \frac{1}{2^n} \sum_{y \in \cube{n}} \chi_z(y) u^2(d(x, y))
    &\overset{\eqref{EQ:q2expansion}}{=}  \frac{1}{2^n} \sum_{y \in \cube{n}} \chi_z(y) \sum_{i = 0}^{2r} \lambda_i \kraw_i(d(x,y)) \\
    &\overset{\eqref{EQ:krawcharacter}}{=} \sum_{i = 0}^{2r} \lambda_i \sum_{y \in \cube{n}} \chi_z(y) \sum_{|a| = i} \chi_{a}(x) \chi_{a}(y) \\
    &= \sum_{i = 0}^{2r} \lambda_i \sum_{|a| = i} \Big(\sum_{y \in \cube{n}} \chi_z(y) \chi_{a}(y) \Big) \chi_{a}(x) \\
    &\overset{\eqref{EQ:characterinner}}{=} \frac{1}{2^n} \sum_{i = 0}^{2r} \lambda_i \sum_{|a| = i} 2^n \delta_{z, a}  \chi_{a}(x) \\
    &= \lambda_{|z|} \chi_z(x).
\end{align*} \qedMP
\end{proof}
Finally, we note that since the Krawtchouk polynomials form an \emph{orthogonal} basis for $\R[t]$, we may express the coefficients $\lambda_i$ in the decomposition  (\ref{EQ:q2expansion}) of $u^2$ in the following way:
\begin{equation}
    \label{eq:krawnormlambda}  
      \lambda_i = \langle \kraw_i, u^2 \rangle_\omega ~/~ \|\kraw_i\|^2_\omega = \langle \krawnorm_i, u^2 \rangle_\omega.
\end{equation}
In addition, since in view of Lemma \ref{LEM:kraw1} we have $\krawnorm_i(t)\le \krawnorm_0(t)$ for all $t\in [0:n]$, it folllows that 
\begin{equation}\label{eqlambda}
\lambda_i \le \lambda_0 \quad \text{ for } 0\le i\le 2r.
\end{equation}

\section{The inner Lasserre hierachy and orthogonal polynomials}
\label{SEC:innerbound}

The inner Lasserre hierachy, which we have  defined for the boolean cube in \eqref{EQ:upbound},  may be defined more generally for the minimization of a polynomial $g$ over a compact set $M \subseteq \R^n$ equipped with a measure $\nu$ with support $M$, by setting:
\begin{equation}
\label{EQ:generalupbound}
	g^{(r)} := g^{(r)}_{M, \nu} = \inf_{s \in \Sigma_r} \left\{ \int_M {g \cdot s d\nu} : \int_M s d\nu = 1 \right\}
\end{equation}
for any integer $r\in \N$. So we have: $g^{(r)}\ge g_{\min}:=\min_{x\in M}g(x).$
A crucial ingredient of the proof of our main theorem below is an analysis of the error $g^{(r)} - g_{\min}$ in this hierarchy for a special choice of  $M \subseteq \R$, $g$ and $\nu$.

Here, we recall a technique which may be used to perform such an analysis in the univariate case, which was developed in \cite{deKlerkLaurent2018} and further employed for this purpose, e.g., in \cite{deKlerkLaurent2019,SlotLaurent2019}.

First, we observe that we may always replace $g$ by a suitable \emph{upper estimator} $\widehat g$ which satisfies $\widehat g_{\min} = g_{\min}$ and $\widehat g(x) \geq g(x)$ for all $x \in M$. Indeed, it is clear that for such $\widehat g$ we have:
\[
	g^{(r)} - g_{\min} \leq \widehat g^{(r)} - g_{\min} = \widehat g^{(r)} - \widehat g_{\min}.
\]

Next, we consider the special case when $M\subseteq \mathbb R$ and $g(t) = t$. Here, the bound $g^{(r)}$ may be expressed in terms of the orthogonal polynomials on $M$ w.r.t. the measure $\nu$, i.e., the polynomials $p_i \in \R[t]_i$ determined by the relation:
\[
	\int_M p_i p_j d\nu = 0 \text{ if } i \neq j.
\]
\begin{theorem}[\cite{deKlerkLaurent2018}]
\label{THM:deKlerkLaurentinner}
Let $M \subseteq \R$ be an interval 
and let $\nu$ be a measure supported on $M$ with corresponding orthogonal polynomials $p_i\in\R[t]_i$ ($i \in \N$). Then the Lasserre inner bound $g^{(r)}$ (from (\ref{EQ:generalupbound})) of order $r$ for the polynomial $g(t) = t$  equals
\[
	g^{(r)} = \xi_{r+1},
\]
where $\xi_{r+1}$ is the smallest root of the polynomial $p_{r+1}$.
\end{theorem}

\begin{remark}
\label{REM:linearunivariate}
The upshot of the above result is that, for any polynomial $g : \R \rightarrow \R$ which is upper bounded on an interval $M\subseteq \R$ by a linear polynomial $\widehat g(t) = ct$ for some $c > 0$, we have:
\begin{equation}
	\label{EQ:upboundleastroot}
	g^{(r)} - g_{\min} \leq c \cdot \xi_{r+1},
\end{equation}
where $\xi_{r+1}$ is the smallest root of the corresponding orthogonal polynomial of degree $r+1$.
\end{remark}

\section{Proof of Theorem \ref{THM:main}}
\label{SEC:mainproof}

Throughout, $d\le n$ is a fixed integer (the degree of the polynomial $f$ to be minimized over $\cube{n}$). Recall $u \in \R[t]$ is a univariate polynomial with degree $r$ (that we need to select) with $2r\ge d$.
Consider  the associated kernel $\kernel$ defined in \eqref{EQ:kernel} and the corresponding linear operator $\kernelop$ from (\ref{eqKp}).
 Recall from \eqref{EQ:overview1} that we are interested in bounding the quantity:
\begin{equation*}
	\|\kernelop^{-1} - I\| := \sup_{p \in \R[x]_d} \frac{\| \kernelop^{-1} p - p \|_\infty}{\|p\|_\infty}.
\end{equation*}
Our proof consists of two parts. First, we relate the coefficients $\lambda_i$, that appear in  the decomposition \eqref{EQ:q2expansion} of $u^2(t) = \sum_{i = 0}^{2r} \lambda_i \kraw_i(t)$ into the basis of Krawtchouk polynomials, to the quantity $\|\kernelop^{-1} - I\|$.

Then, using this relation and the connection between the \emph{inner} Lasserre hierarchy and extremal roots of orthogonal polynomials outlined in Section \ref{SEC:innerbound}, we show that $u$ may be chosen such that $\|\kernelop^{-1}-I\|$ is of the order $\xi^n_{r+1} / n$, with $\xi^n_{r+1}$ the smallest root of the degree $r+1$ Krawtchouk polynomial (with parameter $n$).

\subsection*{Bounding $\|\kernelop^{-1}-I\|$ in terms of the $\lambda_i$}
We need the following technical lemma, which bounds the sup-norm $\|p_k\|_\infty$ of the harmonic components $p_k$ of a polynomial $p \in \redpolys{x}$ in terms of $\|p\|_\infty$, the sup-norm of $p$ itself. The key point is that this bound is independent of the dimension $n$. We delay the proof which is rather technical to Section~\ref{APP:harmcompbound}.
\begin{lemma}
\label{LEM:cube}
There exists a constant $\harmbound_d > 0$, depending only on $d$, such that for any $p = p_0 + p_1 + \ldots + p_d \in \redpolys{x}_d$, we have:
\[
    \| p_k \|_\infty \leq \harmbound_d \|p\|_\infty \text{ for all } 0\le k \leq d.
\]
\end{lemma}

\smallskip
\begin{corollary}
Assume that $\lambda_0 = 1$ and $\lambda_k \ne 0$ for $1\le k\le d$. Then we have:
\begin{equation}
	\label{EQ:nonlinlamb}
    \|\kernelop^{-1} - I\| \leq \harmbound_d \cdot \nonlinlamb, \text{ where } \nonlinlamb := \sum_{i = 1}^d |\lambda_i^{-1} - 1|.
\end{equation}
\end{corollary}

\begin{proof}
By assumption, the operator $\kernelop$ is invertible and, in view of Funk-Hecke relation (\ref{EQ:FunkHecke}),  its inverse is given by
$\kernelop^{-1}p =\sum_{i=0}^d \lambda_i^{-1}p_k$ for 
 any $p = p_0 + p_1 + \ldots + p_d \in \redpolys{x}_d$.
 Then we have:
\begin{equation}
\begin{split}
\label{EQ:kernelharmbound}
\|\kernelop^{-1} p - p\|_\infty = \|\sum_{i=1}^d (\lambda_i^{-1} - 1)p_i \|_\infty &\leq \sum_{i=1}^d |\lambda_i^{-1} - 1| \|p_i\|_\infty \\
&\leq \sum_{i=1}^d |\lambda_i^{-1} - 1| \cdot \harmbound_d \| p \|_\infty,
\end{split}
\end{equation}
where we use Lemma \ref{LEM:cube} for the last inequality. \qedMP
\end{proof}

The expression $\nonlinlamb$ in \eqref{EQ:nonlinlamb} is difficult to analyze. Therefore, following \cite{FangFawzi2019}, we consider instead the simpler expression:
\begin{equation*}
	\linlamb := \sum_{i=1}^d (1 - \lambda_i) = d - \sum_{i=1}^d \lambda_i,
\end{equation*}
\sloppy
which is linear in the $\lambda_i$. Under the assumption that $\lambda_0 = 1$, we have ${\lambda_i \leq \lambda_0 = 1}$ for all $i$ (recall relation (\ref{eqlambda})).
Thus, $\nonlinlamb$ and $\linlamb$ are both minimized when the $\lambda_i$ are close to $1$. The following lemma makes this precise.

\begin{lemma}
\label{LEM:normbound}
Assume that $\lambda_0 = 1$ and that $\linlamb \leq 1/2$. Then we have $\nonlinlamb \leq 2 \linlamb$, and thus
\begin{equation*}
\|\kernelop^{-1} - I\| \leq 2 \harmbound_d \cdot \linlamb.
\end{equation*}
\end{lemma}
\begin{proof}
As we assume $\linlamb \leq 1/2$, we must have $1/2 \leq \lambda_i \leq 1$ for all $i$. Therefore, we may write:
\[
\nonlinlamb = \sum_{i = 1}^d |\lambda_i^{-1} - 1| = \sum_{i = 1}^d |(1 - \lambda_i) / \lambda_i| = \sum_{i = 1}^d (1 - \lambda_i) / \lambda_i \leq 2\sum_{i = 1}^d (1 - \lambda_i) = 2 \linlamb.
\] \qedMP
\end{proof}

\subsection*{Optimizing the choice of the univariate polynomial $u$}
In light of Lemma \ref{LEM:normbound}, and recalling \eqref{eq:krawnormlambda}, we wish to find a univariate polynomial $u \in \R[t]_r$ for which
\begin{align*}
    \lambda_0 &= \langle 1, u^2 \rangle_\omega = 1, \text{ and }\\
    \linlamb = d - \sum_{i = 1}^d \lambda_i &= d - \sum_{i=1}^d \langle \krawnorm_i, u^2 \rangle_\omega \text{ is small.}
\end{align*}
Unpacking the definition of $\langle \cdot, \cdot \rangle_{\omega}$, we thus need to solve the following optimization problem:
\begin{equation}
\label{EQ:qopt}
    \inf_{u \in \R[t]_r} \left\{\int g \cdot u^2 d\omega : \int u^2 d \omega = 1 \right\}, \text{ where } g(t) := d - \sum_{i=1}^d \krawnorm_i(t).
\end{equation}
(Indeed $\int gu^2d\omega =\langle g,u^2\rangle_\omega= \tilde \Lambda$ and $\int u^2d\omega= \langle 1,u^2\rangle_\omega$.)
We recognize this program to be the same as the program \eqref{EQ:generalupbound} defining the inner Lasserre bound\footnote{Technically, the program \eqref{EQ:generalupbound} allows for the density to be a \emph{sum of squares}, whereas the program \eqref{EQ:qopt} requires the density to be an actual square. This is no true restriction, though, since, as a straightforward convexity argument shows, the optimum solution to \eqref{EQ:generalupbound} can in fact always be chosen to be a square  \cite{Lasserre2010}.}  of order $r$ for the minimum $g_{\min} = g(0) = 0$ of the polynomial $g$ over $[0 : n]$, computed with respect to the measure $d\omega(t) = 2^{-n}{n \choose t}$. Hence the optimal value of (\ref{EQ:qopt}) is equal to $g^{(r)}$ and, using Lemma \ref{LEM:normbound}, we may conclude the following.

\begin{theorem}
\label{THM:innerlasbound}
Let $g$ be as in \eqref{EQ:qopt}. Assume that $g^{(r)} - g_{\min} \leq 1/2$. Then  there exists a polynomial $u \in \R[t]_r$ such that  $\lambda_0 = 1$ and
\[
\| \kernelop^{-1} - I \| \leq 2 \harmbound_d \cdot (g^{(r)} - g_{\min}).
\]
Here, $g^{(r)}$ is the inner Lasserre bound on $g_{\min}$ of order $r$, computed on $[0, n]$ w.r.t. $\omega$, via the program (\ref{EQ:qopt}), and  $\harmbound_d$ is the constant of Lemma \ref{LEM:cube}.
\end{theorem}

It remains, then, to analyze the range $g^{(r)} - g_{\min}$. For this purpose, we follow the technique outlined in Section \ref{SEC:innerbound}.
We first show that $g$ can be upper bounded by its linear approximation at $t = 0$.
\begin{lemma}
\label{LEM:kraw2}
We have:
\[
g(t) \leq \widehat g(t) := d(d+1) \cdot (t/n) \quad \forall t \in [0:n].
\]
Furthermore, the minimum $\hat g_{\min}$ of $\hat g$ on $[0 : n]$ clearly satisfies $\widehat g_{\min} = \widehat g(0) = g(0) = g_{\min}$.
\end{lemma}
\begin{proof}
Using \eqref{EQ:krawabsdist}, we find for each $k \leq n$ that:
\[
\krawnorm_k(t) \geq \krawnorm_k(0) - \frac{2k}{n} \cdot t = 1 - \frac{2k}{n} \cdot t  \quad \forall t \in [0 : n].
\]
Therefore, we have:
\[
g(t) := d - \sum_{k=1}^d \krawnorm_k(t) \leq \sum_{k=1}^d \frac{2k}{n} \cdot t = d(d+1) \cdot (t/n) \quad \forall t \in [0 : n].
\] \qedMP
\end{proof}
\begin{lemma}
\label{LEM:symupboundanalysis}
We have:
\begin{equation}\label{eqgup}
    g^{(r)} - g_{\min} \leq d(d+1) \cdot (\xi_{r+1}^n/ n),
\end{equation}
where $\xi_{r + 1}^n$ is the smallest root of the Krawtchouk polynomial $\kraw_{r+1}(t)$.
\end{lemma}
\begin{proof}
This follows immediately from Lemma \ref{LEM:kraw2} and Remark \ref{REM:linearunivariate} at the end of Section \ref{SEC:innerbound}, noting that the Krawtchouk polynomials are indeed orthogonal w.r.t. the measure $\omega$ on $[0:n]$  (cf. \eqref{EQ:krawortho}). \qedMP
\end{proof}

Putting things together, we may prove our main result, Theorem \ref{THM:main}.

\begin{proof}[of Theorem~\ref{THM:main}]
Assume that $r$ is large enough so that ${d(d+1)\cdot (\xi^n_{r+1}/n) \le 1/2}$. By Lemma~\ref{LEM:symupboundanalysis}, we then have 
\[g^{(r)} - g_{\min} \leq d(d+1) \cdot (\xi_{r+1}^n/ n) \leq 1/2. \]
By Theorem \ref{THM:innerlasbound} we are thus able to choose a polynomial $u \in \R[t]_r$ whose associated operator $\kernelop$ satisfies $\kernelop(1)=1$ and 
\[\|\kernelop^{-1}-I\|\le 2\gamma_d \cdot d(d+1) \cdot  (\xi_{r+1}^n/n).\] We may then use Lemma~\ref{lemK0} to obtain Theorem~\ref{THM:main} with constant ${C_d := \gamma_d \cdot d(d+1)}$. \qedMP
\end{proof}


\section{Proof of Theorem \ref{THM:mainup}}
\label{SEC:generalize}

We turn now to analyzing the inner hierarchy $\upbound{r}$ defined in \eqref{EQ:upbound} for a polynomial $f\in \R[x]_d$ on the boolean cube, whose definition is repeated for convenience:
\begin{equation}
\label{EQ:upbound2}
\upbound{r} := \min_{s \in \Sigma[x]_r} \left\{ \int_{\cube{n}} f(x) \cdot s(x) d\mu : \int_{\cube{n}} s(x) d \mu = 1 \right\} \geq \funcmin.
\end{equation}
As before, we may assume w.l.o.g. that $f_{\min} = f(0) = 0$ and that $f_{\max} = 1$. To facilitate the analysis of the bounds $f^{(r)}$, the idea is to restrict in (\ref{EQ:upbound2}) to polynomials $s(x)$ that are invariant under the action of  $\stab(0) \subseteq \Aut(\cube{n})$, i.e., depending only on the Hamming weight $|x|$. Such polynomials are of the form $s(x)=u(|x|)$ for some univariate polynomial $u\in \R[t]$. Hence this leads to 
the following, weaker hierarchy, where we now optimize over  {\em univariate} sums-of-squares:
\begin{equation*}
\upboundsym{r} := \min_{u\in \Sigma[t]_r} \left\{ \int_{\cube{n}} f(x) \cdot u(|x|) d\mu(x) : \int_{\cube{n}} u(|x|) d \mu(x) = 1 \right\}.
\end{equation*}
By definition, we must have $\upboundsym{r} \geq \upbound{r} \geq \funcmin$, and so an analysis of $\upboundsym{r}$ extends immediately to $\upbound{r}$.

The main advantage of working with the hierarchy $\upboundsym{r}$ is that we may now assume that $f$ is itself invariant under $\stab(0)$, after replacing $f$ by its symmetrization:
\[
	\frac{1}{|\stab(0)|}\sum_{\sigma \in \stab(0)} f(\sigma(x)).
\]
Indeed, for any $u \in \Sigma[t]_r$, we have that:
\begin{align*}
    \int_{\cube{n}} f(x) u(|x|) d\mu(x) 
    &= \frac{1}{|\stab(0)|} \sum_{\sigma \in \stab(0)}\int_{\cube{n}} f(\sigma(x)) u(|\sigma(x)|) d\mu(\sigma(x)) \\
    &= \int_{\cube{n}} \frac{1}{|\stab(0)|} \sum_{\sigma \in \stab(0)} f(\sigma(x)) u(|x|) d\mu(x),
\end{align*}
So we now assume  that $f$ is $\stab(0)$-invariant, and thus we may write:
\[
    f(x) = F(|x|) \text{ for some polynomial  } F(t) \in\R[t]_d.
\]
By the definitions of the measures $\mu$ on $\cube{n}$ and $\omega$ on $[0:n]$ we have the identities:
\begin{align*}
\int_{\cube{n}}u(|x|)d\mu(x) &= \int_{[0:n]}u(t)d\omega(t), \\
\int_{\cube{n}} F(|x|) u(|x|) d\mu(x) &= \int_{[0:n]} F(t)u(t)d\omega(t).
\end{align*}
Hence we get
\begin{equation}
\label{EQ:symmetrizedinner}
	\upboundsym{r} = \min_{u \in \Sigma[t]_r} \Big\{ \int_{[0:n]} F(t) \cdot u(t) d\omega(t) : \int _{[0:n]} u(t) d \omega(t) = 1 \Big\} = F_{[0: n], \omega}^{(r)}.
\end{equation}
In other words, the behaviour of the \emph{symmetrized} inner hierarchy $\upboundsym{r}$ over the boolean cube  w.r.t. the uniform measure $\mu$  is captured by the behaviour of the \emph{univariate} inner hierarchy $F_{[0: n], \omega}^{(r)}$ over $[0 : n]$ w.r.t. the discrete measure $\omega$.

Now, we are in a position to make use again of the technique outlined in Section \ref{SEC:innerbound}. First we 
find a linear upper estimator $\widehat F$ for $F$ on $[0:n]$. 

\begin{lemma}
\label{LEM:Flinearupperestimator}
We have
\[
    F(t) \leq \widehat F(t):= d(d+1) \cdot \gamma_d \cdot t/n
     \quad \forall t \in [0 : n],
\]
where $\gamma_d$ is the same constant as in Lemma \ref{LEM:cube}.
\end{lemma}

\begin{proof}
Write $F(t)= \sum_{i=0}^d \lambda_i \krawnorm_i(t) $ for some scalars $\lambda_i$. 
By assumption, ${F(0)=0}$ and thus $\sum_{i=0}^d\lambda_i=0$. We now use an analogous argument as for Lemma \ref{LEM:kraw2}:
\begin{align*}
F(t)=\sum_{i=0}^d \lambda _i (\krawnorm_i(t)-1) &\le \sum_{i=0}^d |\lambda_i| |\krawnorm_i(t)-1| 
\stackrel{\eqref{EQ:krawabsdist}}{\le} \max_i|\lambda_i| \cdot t \cdot \sum_{i=0}^d \frac{2i}{n} \\
&\le  \max_i|\lambda_i| \cdot t \cdot {d(d+1)\over n}.
\end{align*}
As $\|f\|_\infty =1$, using Lemma \ref{LEM:cube}, we can conclude that: $$|\lambda_i|=\max_{t\in [0:n]} |\lambda_i\krawnorm_i(t)| \le \gamma_d$$ which gives the desired result. \qedMP
\end{proof}

In light of Remark \ref{REM:linearunivariate} in Section~\ref{SEC:innerbound}, we may now conclude that 
\[
F_{[0: n], \omega}^{(r)} \leq d(d+1) \harmbound_d \cdot \xi_{r+1}^n/n.
\]
As $\upbound{r} \leq \upboundsym{r} = F_{[0: n], \omega}^{(r)}$, we have thus shown Theorem~\ref{THM:mainup} with constant ${C_d = d(d+1) \harmbound_d}$. Note that in comparison to Lemma \ref{LEM:symupboundanalysis}, we only have the additional constant factor $\gamma_d$.


\subsection*{Exactness of the inner hierarchy}
As is the case for the outer hierarchy, the inner hierarchy is exact when $r$ is large enough. Whereas the outer hierarchy, however, is exact for $r \geq (n+d-1)/2$, the inner hierarchy is exact in general if and only if $r \geq n$. We give a short proof of this fact below, for reference.

\begin{lemma}
Let $f$ be a polynomial on $\cube{n}$. Then $\upbound{r} = \funcmin~$ for all $r \geq n$.
\end{lemma}
\begin{proof}
We may assume w.l.o.g. that $f(0) = \funcmin$. Consider the interpolation polynomial:
\[
s(x) := \sqrt{2^n} \prod_{i=1}^n (1-x_i) \in \R[x]_n,
\]
which satisfies $s^2(0) = 2^n$ and $s^2(x) = 0$ for all $0 \neq x \in \cube{n}$. Clearly, we have:
\[
    \int f s^2 d\mu = f(0) = \funcmin\  \text{ and }\  \int s^2 d\mu = 1,
\]
and so $\upbound{n} = \funcmin$. \qedMP
\end{proof}

The next lemma shows that this result is tight, by giving an example of polynomial $f$ for which the bound $\upbound{r}$ is exact only at order $r=n$.

\begin{lemma}
Let $f(x) = |x|=x_1+\ldots+x_n$. Then $\upbound{r} - \funcmin > 0$ \ for all $r < n$.
\end{lemma}
\begin{proof}
Suppose not. That is, $f^{(r)}= f_{\min}=0$ for some $r\le n-1$. 
As ${f(x) > 0 = \funcmin}$ for all $0 \neq x \in \cube{n}$, this implies that there exists a polynomial $s \in \R[x]_r$ such that $s^2$ is interpolating at $0$, i.e. such that $s^2(0) = 1$ and $s^2(x) = 0$ for all $0 \neq x \in \cube{n}$. But then $s$ is itself interpolating at $0$ and has degree $r < n$, a contradiction. \qedMP
\end{proof}

\section{Proof of Lemma \ref{LEM:cube}}
\label{APP:harmcompbound}

In this section we give a proof of Lemma \ref{LEM:cube}, where we bound the sup-norm $\| p_k \|_{\infty}$ of the harmonic components $p_k$ of a polynomial $p$ 
by $\gamma_d \|p\|_\infty$ for some constant $\gamma_d$ depending only on the degree $d$ of $p$. 
The following definitions will be convenient.
\begin{definition}
For $n \geq  d \geq k \geq 0$ integers, we write:
\begin{align*}
    \rho(n, d, k) &:= \sup \{\|p_k\|_\infty : p = p_0 + p_1 + \dots + p_d \in \redpolys{x}_d, \|p\|_{\infty} \leq 1 \}, \text{ and }\\
    \rho(n, d) &:= \max_{0 \leq k \leq d} \rho(n, d, k).
\end{align*}
We are thus interested in finding a bound $\harmbound_d$ depending only on $d$ such that:
\begin{equation}\label{EQ:gammad}
\harmbound_d \geq \rho(n, d) \text{ for all } n \in \N.
\end{equation}
\end{definition}
\noindent We will now show that in the computation of the parameter $\rho(n,d,k)$ we may restrict to feasible solutions $p$ having strong structural properties.
 First, we show that we may assume that the sup-norm of the harmonic component $p_k$ of $p$ is attained at $0 \in \cube{n}$.

\begin{lemma}
\label{LEM:claimsimp1}
We have
\begin{equation}
\label{EQ:claimsimp1}
    \rho(n, d, k) = \sup_{p \in \redpolys{x}_d^n} \left\{p_k(0) : \|p\|_\infty \leq 1 \right\}    
\end{equation}
\end{lemma}

\begin{proof}
Let $p$ be a feasible 
solution for $\rho(n, d, k)$ and let $x \in \cube{n}$ for which $p_k(x) = \|p_k\|_\infty$ (after possibly replacing $p$ by $-p$). Now choose  $\sigma \in \Aut(\cube{n})$ such that $\sigma(0) = x$ and set $\widehat p = p \circ \sigma$. Clearly, $\widehat p$ is again a feasible solution for $\rho(n, d, k)$. Moreover,  as $H_k$ is invariant under $\Aut(\cube{n})$, we have:
\[
\|\widehat p_k\|_\infty= \widehat p_k(0) = (p \circ \sigma)_k(0) = (p_k \circ \sigma)(0) = \|p_k\|_\infty,
\]
which shows the lemma. \qedMP 
\end{proof}
\noindent
Next we show that we may in addition restrict to polynomials that are highly symmetric.

\begin{lemma}
\label{LEM:claimsimp2}
In the program   \eqref{EQ:claimsimp1}  we may restrict the optimization to polynomials of the form
\[
    p(x) = \sum_{i = 0}^d \lambda_i \sum_{|a| = i} \chi_a(x) = \sum_{i = 0}^d \lambda_i \kraw_i(|x|)\quad \text{ where } \lambda_i\in \R.
\]
\end{lemma}
\begin{proof}
Let $p$ be a feasible solution to \eqref{EQ:claimsimp1}. 
Consider the following polynomial $\widehat p$ obtained as symmetrization of $p$ under action of $\stab(0)$, the set of automorphism of $\cube{n}$ corresponding to the coordinate permutations:
\[
    \widehat p(x) = \frac{1}{|\stab(0)|} \sum_{\sigma \in \stab(0)} (p \circ \sigma) (x).
\]
Then $\|\widehat p\|_\infty\le 1$  and $\widehat p_k(0) = p_k(0)$, so $\widehat p$ is still feasible for \eqref{EQ:claimsimp1} and has the same objective value as $p$. Furthermore,  for each $i$, $\widehat p_i$ is invariant under $\stab(0)$, which implies that $\widehat p_i(x) = \lambda_i \charsum_i(x) = \lambda_i \sum_{|a| = i} \chi_a(x)=\lambda_i \kraw_i(|x|)$ for some $\lambda_i \in \R$ (see \eqref{EQ:charsum}). \qedMP
\end{proof}
A simple rescaling $\lambda_i \gets \lambda_i \cdot {n \choose i}$ allows us to switch from $\kraw_i$ to $\krawnorm_i=\kraw_i/{n\choose i}$
and to obtain the following reformulation of $\rho(n, d, k)$ as a linear program.
\begin{lemma}
For any $n\ge d\ge k$ we have:
\begin{equation}
\label{EQ:primal}
    \begin{split}
    \rho(n, d, k) 
    = \max \quad &\lambda_k \\
    s.t. \quad & -1 \leq \sum_{i=0}^d \lambda_i \krawnorm_i(t) \leq 1 \quad (t=0,1, \dots, n).  
    \end{split}
\end{equation}
\end{lemma}

\subsection*{Limit functions}
The idea now is to prove a bound on $\rho(n, d, d)$ 
which holds for  fixed $d$ and is  independent of $n$. We will do this  by considering `the limit' of problem \eqref{EQ:primal} as $n \to \infty$. 
For each $k \in \N$, we define the limit function: 
\begin{equation*}
    \krawlim_k(t) := \lim_{n \to \infty} \krawnorm_{k}(nt),
\end{equation*}
which, as shown in Lemma \ref{lemlimit} below, is in fact a polynomial. We first present  the polynomial $\krawlim_k(t)$ for small $k$ as an illustration.
\begin{example}
We have:
\begin{align*}
    \krawnorm_{0}(nt) = 1 &\implies \krawlim_{0}(t) = 1, &\\
    \krawnorm_{1}(nt) = -2t + 1 &\implies \krawlim_{1}(t) = -2t + 1, &\\
    \krawnorm_{2}(nt) = \frac{2n^2t^2 - 2n^2t + {n \choose 2}}{{n \choose 2}} &\implies \krawlim_{2}(t) = 4t^2 - 4t + 1=(1-2t)^2.& \\
\end{align*}
\end{example}

\begin{lemma}\label{lemlimit}
We have: $\krawlim_k(t) = (1 - 2t)^k$ for all $k \in \N$.
\end{lemma}
\begin{proof}
The Krawtchouk polynomials satisfy the following three-term recurrence relation (see, e.g., \cite{MacwilliamsSloane1983}):
\[
	(k+1)\kraw_{k+1}(t) = (n - 2t) \kraw_{k}(t) - (n - k + 1) \kraw_{k-1}(t)
\]
for $1\le k\le n-1$. By evaluating the polynomials at $nt$ we obtain:
\begin{align*}
	     &(k+1)\kraw_{k+1}(nt) = (n - 2nt) \kraw_{k}(nt) - (n - k + 1) \kraw_{k-1}(nt), \\
\implies \quad &(k+1){n \choose k+1}\krawnorm_{k+1}(nt) = (n - 2nt) {n \choose k}\krawnorm_{k}(nt) - (n - k + 1){n \choose k-1} \krawnorm_{k-1}(nt), \\
\implies \quad &\krawnorm_{k+1}(nt) = {n(1 - 2t)\over (n-k)} \cdot \krawnorm_{k}(nt) - {k\over n-k}  \cdot \krawnorm_{k-1}(nt), \\
\implies \quad &\krawlim_{k+1}(t) = (1-2t) \krawlim_k(t).
\end{align*}
As $\krawlim_{0}(t) = 1$ and $\krawlim_{1}(t) = 1 - 2t$ we can conclude that indeed ${\krawlim_k(t) = (1 - 2t)^k}$ for all $k \in \N$. \qedMP
\end{proof}
Next, we show that solutions to \eqref{EQ:primal} remain feasible after increasing the dimension $n$.

\begin{lemma}
\label{LEM:solextension}
Let $\lambda = (\lambda_0, \lambda_1, \dots, \lambda_d)$ be a feasible solution to $\eqref{EQ:primal}$ for a certain $n \in \N$. Then it is also feasible to $\eqref{EQ:primal}$ for $n+1$ (and thus for any $n'\ge n+1$).
Therefore, $\rho(n+1,d,k) \ge \rho(n,d,k)$ for all $n\ge d\ge k$ and thus $\rho(n+1,d)\ge \rho(n,d)$ for all $n\ge d$.
\end{lemma}

\begin{proof}
We may view $\cube{n}$ as a subset of $\cube{n+1}$ via the map $a \mapsto (a,0)$, and analogously $\redpolys{x_1,\ldots,x_n}$ as a subset of $ \redpolys{x_1,\ldots,x_n,x_{n+1}}$ via $\chi_a \mapsto \chi_{(a,0)}$. Now for $m, i \in \N$ we consider again the zonal spherical harmonic \eqref{EQ:charsum}:
\[
    \charsum_i^m = \sum_{|a|= i, a \in \cube{m}} \chi_a.
\]
Consider the set $\stab(0) \subseteq \Aut(\cube{n+1})$ of automorphisms fixing $0 \in \cube{n+1}$, i.e., the coordinate permutations arising from $\sigma\in \Sym(n+1)$. We will use the following identity:
\begin{equation}
\label{EQ:solextension}
\frac{1}{|\stab(0)|} \sum_{\sigma \in \stab(0)} \frac{X_i^n}{{n \choose i}} \circ \sigma = \frac{X_i^{n+1}}{{n+1 \choose i}}.
\end{equation}
To see that \eqref{EQ:solextension} holds note that its left hand side is equal to
$${1\over (n+1)!{n\choose i}} \sum_{\sigma\in \Sym(n+1)} \sum_{a\in \cube{n}, |a|=i} \chi_{(a,0)}\circ \sigma
= {1\over (n+1)!{n\choose i}}\sum_{b\in \cube{n+1}, |b|=i} N_b \chi_b,
$$
where $N_b$ denotes the number of pairs $(\sigma,a)$ with $\sigma\in \Sym(n+1)$, $a\in \cube{n}$, $|a|=i$ such that $b=\sigma(a,0)$.
As there are $n\choose i$ choices for $a$ and $i!(n + 1 - i)!$ choices for $\sigma$ we have $N_b= {n\choose i}i!(n + 1 - i)!$ and thus \eqref{EQ:solextension} holds.

Assume $\lambda$ is a feasible solution of \eqref{EQ:primal} for a given value of $n$. Then, in view of \eqref{EQ:krawcharacter}, this means
\[
    \Big|\sum_{i=0}^d \lambda_i \cdot \frac{X_i^n(x)}{{n \choose i}}\Big| \leq 1 \quad \text{ for all } x\in \cube{n}, \quad \text{ and thus  for all } x \in \cube{n + 1}.
\]
Using \eqref{EQ:solextension} we obtain:
\begin{align*}
\Big|\sum_{i=0}^d \lambda_i \frac{X_i^{n+1}(x)}{{n+1 \choose i}}\Big|
&= \Big|\sum_{i=0}^d \lambda_i \cdot \frac{1}{|\stab(0)|} \sum_{\sigma \in \stab(0)} \frac{X_i^n(\sigma(x))}{{n \choose i}}\Big| \\
&= \Big| \bigg(\frac{1}{|\stab(0)|} \sum_{\sigma \in \stab(0)} \big(\sum_{i=0}^d \lambda_i \frac{X_i^n}{{n \choose i}}\big) \circ \sigma\bigg)(x)\Big|\leq 1
\end{align*}
for all $x\in \cube{n+1}$. Using \eqref{EQ:krawcharacter} again, this shows  that $\lambda$ is a feasible solution of  program  \eqref{EQ:primal} for $n+1$. \qedMP
\end{proof}

\begin{example}
To illustrate the identity \eqref{EQ:solextension}, we give a small example with $n=i=2$. Consider:
\[
    X^2_2 = \sum_{|a| = 2, a \in \cube{2}} \chi_a = \chi_{11}.
\]
The automorphisms in $\stab(0) \subseteq \Aut(\cube{3})$ fixing $0 \in \cube{3}$ are  the permutations of $x_1, x_2, x_3$. So we get:
\begin{align*}
    \frac{1}{|\stab(0)|} \sum_{\sigma \in \stab(0)}  X^2_2 \circ \sigma &= \frac{1}{6}(\chi_{110} + \chi_{101} + \chi_{110} + \chi_{011} + \chi_{101} + \chi_{011}) \\
    &= \frac{2}{6}(\chi_{110} + \chi_{101} + \chi_{011}) = \frac{1}{3}X^3_2,
\end{align*}
and indeed ${2 \choose 2} / {3 \choose 2} = 1/3$.
\end{example}

\begin{lemma}
\label{LEM:sollim}
For $d\ge k \in \N$, define  the program:
\begin{equation}
    \label{EQ:primallim}
    \begin{split}
    \rho(\infty, d, k) := \max \quad &\lambda_k \\
    s.t. \quad &-1 \leq \sum_{i=0}^d \lambda_i \krawlim_i(t) \leq 1 \quad (t \in [0, 1]).
    \end{split}
\end{equation}
Then, for any $n\ge d$,  we have: $\rho(n, d, k) \leq \rho(\infty, d, k)$.
\end{lemma}

\begin{proof}
Let $\lambda$ be a feasible solution to $\eqref{EQ:primal}$ for $(n, d, k$). We show that $\lambda$ is feasible for \eqref{EQ:primallim}. For this fix  $t \in [0, 1] \cap \Q$. Then there exists a sequence of integers $(n_j)_j \to \infty$ such that $n_j \geq n$ and $tn_j \in [0, n_j]$ is integer for each $j \in \N$. As $n_j\ge n$, we know from Lemma \ref{LEM:solextension} that $\lambda$ is also a feasible solution of program $\eqref{EQ:primal}$ for $(n_j,d,k)$. Hence, since $n_jt\in [0:n_j]$  we obtain
\[
    |\sum_{i=0}^d \lambda_i \krawnormnosup_{i}^{n_j}(n_j t)| \leq 1 \quad \forall j \in \N.
\]
But this immediately gives:
\begin{equation}
    \label{EQ:feasiblelim}
    |\sum_{i=0}^d \lambda_i \krawlim_{i}(t)| = \lim_{j \to \infty} |\sum_{i=0}^d \lambda_i \krawnormnosup_{i}^{n_j}(n_j t)| \leq 1.
\end{equation}
As $[0, 1] \cap \Q$ lies dense in $[0, 1]$ (and the $\krawlim_i$'s are continuous) we may conclude that \eqref{EQ:feasiblelim} holds for all $t \in [0, 1]$. 
This shows that  $\lambda$ is feasible for \eqref{EQ:primallim}
and we thus have $\rho(n, d, k) \leq \rho(\infty, d, k)$, as desired. \qedMP
\end{proof}
It remains now to compute the optimum solution to the program \eqref{EQ:primallim}. In light of Lemma \ref{lemlimit}, and after a change of variables $x = 1 - 2t$, this program may be reformulated as:
\begin{equation}
    \label{EQ:monomialprogram}
    \begin{split}
    \max \quad & |\lambda_k| \\
    s.t. \quad &-1 \leq \sum_{i=0}^d \lambda_i x^i \leq 1 \quad (x \in [-1, 1]).
    \end{split}
\end{equation}
In other words, we are tasked with finding a polynomial $p(x)$ of degree $d$ satisfying $|p(x)| \leq 1$ for all $x \in [-1, 1]$, whose $k$-th coefficient is as large as possible in absolute value. This is a classical extremal problem solved by V. Markov.
\begin{theorem}[see, e.g., Theorem 7, pp. 53 in \cite{Natanson1964}]
\label{THM:Markovextremal}
For $m \in \N$, let $T_m(x) = \sum_{i=0}^m t_{m, i} x^i$ be the Chebyshev polynomial of degree $m$. Then the optimum solution $\lambda$ to \eqref{EQ:monomialprogram} is given by:
\[
	\sum_{i=0}^d \lambda_i x^i = 
	\begin{cases} 
	T_{d}(x) \quad & \text{if } k \equiv d \mod 2, \\
	T_{d-1}(x) \quad & \text{if } k \not\equiv d \mod 2.
	\end{cases}	
\]
In particular, $\rho(\infty, d, k)$ is equal to $|t_{d, k}|$ (resp. $|t_{d-1, k}|$).
\end{theorem}
As the coefficients of the Chebyshev polynomials are known explicitely, Theorem \ref{THM:Markovextremal} allows us to give exact values of the constant $\gamma_d$ appearing in our main results (see Table \ref{TAB:harmbound}). Using the following identity:
\[
	\sum_{i=0}^d |t_{d, i}| = \frac{1}{2}(1+\sqrt{2})^d + \frac{1}{2}(1 - \sqrt{2})^d \leq (1+\sqrt{2})^d,
\]
we are also able to concretely estimate:
\[
	\gamma_d \leq \max_{k \leq d} \rho(\infty, d, k) \leq (1+\sqrt{2})^d.
\]

\section{Concluding remarks}
\subsection*{Summary}
We have shown a theoretical guarantee on the quality of the sum-of-squares hierarchy $\lowbound{r} \leq \funcmin$ for approximating the minimum of a polynomial $f$ of degree $d$ over the boolean cube $\cube{n}$. As far as we are aware, this is the first such analysis that applies to values of $r$ smaller than $(n+d) / 2$, i.e., when the hierarchy is not exact. 
Additionally, our guarantee applies to a second, measure-based hierarchy of bounds $\upbound{r} \geq \funcmin$. Our result may therefore also be interpreted as bounding the range $\upbound{r} - \lowbound{r}$. Our analysis also applies to polynomial optimization over the cube $\{\pm 1\}^n$ (by a simple change of variables), over the $q$-ary cube (see Appendix \ref{APP:qary}) and in the setting of matrix-valued polynomials (see Appendix \ref{APP:matrixvalued}).

\subsection*{Analysis for small values of $r$}
A limitation of Theorem \ref{THM:main} is that the analyis of $\lowbound{r}$ applies only for choices of $d, r, n$ satisfying $d(d+1) \xi^n_{r+1} \leq 1/2$. One may partially avoid this limitation by proving a slightly sharper version of Lemma~\ref{LEM:normbound}, showing instead that
$
	\nonlinlamb \leq \linlamb/(1 - \linlamb),
$
assuming now only that $\linlamb \leq 1$. Indeed, Lemma \ref{LEM:normbound} is a special case of this result, assuming that $\linlamb \leq 1/2$ to obtain $\nonlinlamb \leq 2 \linlamb$. 
Nevertheless, our methods exclude values of $r$ outside of the regime $r = \Omega(n)$.

\subsection*{The constant $\harmbound_d$}
The strength of our results depends in large part on the size of the constant $C_d$ 
appearing in Theorem~\ref{THM:main} and Theorem \ref{THM:mainup}, where we may set $C_d=d(d+1)\gamma_d$. Recall that $\gamma_d$ is defined in Lemma~\ref{LEM:cube} as a constant for which $\|p_k\|_{\infty} \leq \gamma_d \|p\|_{\infty}$ for any polynomial ${p = p_0 + p_1 + \ldots + p_d}$ of degree $d$ and $k \leq d$ on $\cube{n}$, independently of the dimension $n$. In Section~\ref{APP:harmcompbound} we have shown the existence of such a constant. Furthermore, we have shown there that we may choose $\gamma_d \leq (1 + \sqrt{2})^d$, and have given an explicit expression for the smallest possible value of $\gamma_d$ in terms of the coefficients of Chebyshev polynomials. Table \ref{TAB:harmbound} lists these values for small $d$.
\begin{table}[h!]
    \centering
    \def\arraystretch{1.3}
    \begin{tabular}{r|ccccccccccc}
    $d$ 			& $1$ & $2$ & $3$ & $4$ & $5$ & $6$ & $7$ & $8$ & $9$ & $10$ \\ 
    \hline
    $\harmbound_d$ 	& $1$ & $2$ & $4$ & $8$ & $20$ & $48$ & $112$ & $256$ & $576$ & $1280$
    \end{tabular}
    \caption{Values of the constant $\harmbound_d$.}
   	\label{TAB:harmbound}
\end{table}

\subsection*{Computing extremal roots of Krawtchouk polynomials} 
Although Theorem \ref{THM:Levenshtein} provides only an asymptotic bound on the least root $\xi_{r}^n$ of $\kraw_r$, it should be noted that $\xi_{r}^n$ can be computed explicitely for small values of $r, n$, thus allowing for a concrete estimate of the error of both Lasserre hierarchies via Theorem \ref{THM:main} and Theorem \ref{THM:mainup}, respectively. 
Indeed, as is well-known, the root $\xi_{r+1}^n$ is equal to the smallest eigenvalue of the  $(r+1)\times (r+1)$ matrix $A$ (aka Jacobi matrix), whose entries are given by $A_{i,j} = \langle t\krawnorm_i(t), \krawnorm_j(t) \rangle_\omega$ for $i,j\in \{0,1,\ldots, r\}$. See, e.g., \cite{Szego1959} for more details.

\subsection*{Connecting the hierarchies}
Our analysis of the \emph{outer} hierarchy $\lowbound{r}$ on $\cube{n}$ relies essentially on knowledge of the  the \emph{inner} hierarchy $\upbound{r}$. Although not explicitely mentioned there, this is the case for the analysis on $S^{n-1}$ in \cite{FangFawzi2019} as well. As the behaviour of $\upbound{r}$ is generally quite well understood, this suggests a potential avenue for proving further results on $\lowbound{r}$ in other settings.

For instance, the inner hierarchy $\upbound{r}$ is known to converge at a rate in $O(1/r^2)$ on the unit ball $B^n$ or the unit box $[-1, 1]^n$, but matching results on the outer hierarchy $\lowbound{r}$ are not available. The question is thus whether the strategy used for the hypersphere  $S^{n-1}$ in \cite{FangFawzi2019} and for the boolean cube  $\cube{n}$ here might be extended to these cases as well.

A difficulty is that the sets $B^n$ and $[-1, 1]^n$ have a more complicated symmetric structure than $S^{n-1}$ and $\cube{n}$, respectively. In particular, the group actions have uncountably many orbits in these cases, and a direct analog of the Funk-Hecke formula \eqref{EQ:FunkHecke} is not available. New ideas are therefore needed to define the kernel $\kernel(x,y)$ (cf. \eqref{EQ:kernelform}) and analyze its eigenvalues.


\subsection*{Acknowledgments}
We wish to thank Sven Polak and Pepijn Roos Hoefgeest for several useful discussions. We also thank the anonymous referees for their helpful comments and suggestions.

\appendix

\section{The $q$-ary cube}
\label{APP:qary}
In this section, we indicate how our results for the boolean cube $\cube{n}$ may be extended to the $q$-ary cube $\qcube{n} = \{0, 1, \dots, q-1\}^n$ when $q > 2$ is a fixed integer. Here $\Z / q\Z$ denotes the cyclic group of order $q$, so that $\qcube{n}=\cube{n}$ when $q=2$. The lower bound $\lowbound{r}$ for the minimum of a polynomial $f$ over $\qcube{n}$ is defined analogously to the case $q=2$; namely we set
\[
	    \lowbound{r} := \sup_{\lambda \in \R} \left\{\lambda: f(x) - \lambda \text{ is a sum-of-squares of degree at most } 2r \text{ on } \qcube{n} \right\},
\]
where the condition means that $f(x) - \lambda$ agrees with a sum of squares $s \in \Sigma[x]_{2r}$ for all $x \in \qcube{n}$ or, alternatively, that $f - \lambda - s$ belongs to the ideal generated by the polynomials $x_i(x_i -1)\ldots(x_i - q + 1)$ for $i \in [n]$. 
Similarly, the upper bound $\upbound{r}$ is defined as in \eqref{EQ:upbound} after equipping $\qcube{n}$ with the uniform measure $\mu$.
The parameters $\lowbound{r}$ and $\upbound{r}$ may again be computed by solving a semidefinite program of size polynomial in $n$ for fixed $r, q \in \N$, see \cite{Laurentfinite}.

As before $d(x,y)$ denotes the Hamming distance and $|x|$ denotes the Hamming weight (number of nonzero components).  Note that, for $x,y\in\qcube{n}$, $d(x,y)$ can again be expressed as a polynomial in $x,y$, with degree $q-1$ in each of $x$ and  $y$.
 
We will prove Theorem \ref{theoq} below, which can be seen as an analog of Corollary \ref{cor2} for $\qcube{n}$.
The general structure of the proof is identical to that of the case $q=2$. We therefore only give generalizations of arguments as necessary. For reasons that will become clear later, it is most convenient to consider the sum-of-squares bound $\lowbound{r}$ on the minimum $\funcmin$ of a polynomial $f$ with degree at most $(q-1)d$ over $\qcube{n}$, where $d\le n$ is fixed.

\subsection*{Fourier analysis on $\qcube{n}$ and Krawtchouk polynomials}
Consider the space
\[
	\redpolys{x} := \C[x] /  (x_i(x_i-1) \ldots (x_i-q+1): i\in [n])    
\]
consisting of the polynomials on $\qcube{n}$ with complex coefficients. We equip $\redpolys{x}$ with its natural inner product 
$$\langle f,g\rangle_{\mu}=  \int_{\qcube{n}}f(x)\overline{g(x)}d\mu(x)={1\over q^n} \sum_{x \in \qcube{n}} f(x)\overline{g(x)},$$
where $\mu$ is the uniform  measure on $\qcube{n}$.
  The space $\redpolys{x}$ has dimension $|\qcube{n}| = q^n$ over $\C$ and it is spanned by the polynomials of degree up to $(q-1)n$.
The reason we  now need to work with polynomials with complex coefficients is that the characters have complex coefficients when $q>2$.
 
Let $\psi = e^{2\pi i / q}$ be a primitive $q$-th root of unity.  For $a \in \qcube{n}$, the associated {\em character} $\chi_a \in \redpolys{x}$ is defined by:
\[
	\chi_a(x) = \psi^{a \cdot x}\quad (x\in \qcube{n}). 
\]
So  \eqref{EQ:chardef} is indeed the special case of this definition when $q=2$.
The set of all characters $\{ \chi_a : a \in \qcube{n} \}$ forms an orthogonal basis for $\redpolys{x}$ w.r.t. the above inner product $\langle \cdot, \cdot \rangle_{\mu}$. A character $\chi_a$ can be written as a polynomial of degree  $(q-1) \cdot |a|$ on $\qcube{n}$, i.e., we have $\chi_a \in \redpolys{x}_{(q-1)|a|}$ for all $a \in \qcube{n}$.

As before, we have the direct sum decomposition into pairwise orthogonal subspaces:
\[
	\redpolys{x} = H_0 \perp H_1 \perp \dots \perp H_n,
\]
where $H_i$ is spanned by the set $\{ \chi_a : |a| = i \}$ and $H_i\subseteq\R[x]_{(q-1)i}$.
The components $H_i$ are invariant and irreducible under the action of $\Aut(\qcube{n})$, which is generated by the coordinate permutations and the action of $\Sym(q)$ on individual coordinates. Hence any $p \in \redpolys{x}$ of degree at most  $(q-1)d$ can be (uniquely) decomposed as:
\[
	p = p_0 + p_1 + \dots + p_d \quad (p_i \in H_i).
\]
As before $\stab(0) \subseteq \Aut (\qcube{n})$ denotes the stabilizer of $0 \in \qcube{n}$, which is generated by the coordinate permutations and the permutations in $\Sym(q)$ fixing $0$ in $\{0,1,\ldots,q-1\}$ at any individual coordinate.
We note for later reference that the subspace of $H_i$ invariant under action of $\stab(0)$ 
  is of dimension one, and is spanned by the zonal spherical function:
\begin{equation}
\label{EQ:qzonal}
\charsum_i = \sum_{|a| = i}\chi_a \in H_i.
\end{equation}

The Krawtchouk polynomials introduced in Section \ref{SEC:Fourier} have the following generalization in the $q$-ary setting:
\begin{equation*}
    \kraw_k(t) = \kraw_{k, q} (t) := \sum_{i=0}^k (-1)^i (q-1)^{k-i} {t \choose i} {n-t \choose k-i}.
\end{equation*}
Analogously to relation \eqref{EQ:krawortho}, the Krawtchouk polynomials $\kraw_k$ ($0\le k\le n$) are pairwise orthogonal  w.r.t. the discrete measure $\omega$ on $[0:n]$ given by:
\begin{equation}\label{eqwq}
\omega(t) = \frac{1}{q^n} \sum_{t=0}^n w(t) \delta_t, \text{ with } w(t) := (q-1)^t {n \choose t}.
\end{equation}
To be precise, we have:
\[
	\sum_{t=0}^n \kraw_{k}(t) \kraw_{k'}(t) (q-1)^t {n \choose t} = \delta_{k, k'} (q-1)^k {n \choose k}.
\]
As $\kraw_k(0) = (q-1)^k {n \choose k} = \| \kraw_k \|^2_{\omega}$, we may normalize  $\kraw_k$ by setting:
\[
	\krawnorm_k(t) := \kraw_k(t) / \kraw_k(0) = \kraw_k(t) / \| \kraw_k \|^2_{\omega},
\]
so that $\krawnorm_k$ satisfies $\max_{t=0}^n \krawnorm_k(t) = \krawnorm_k(0) = 1$ (cf. \eqref{EQ:krawnorm}).

We have the following connection (cf. \eqref{EQ:xt}) between the characters and the Krawtchouk polynomials:
\begin{equation}
	\label{EQ:qxt}
	\sum_{a \in \qcube{n} : |a| = k} \chi_a(x) = \kraw_k(i) \quad \text{ for  } x\in \qcube{n} \text{ with } |x| = i.
\end{equation}
Note that for all $a, x, y \in \qcube{n}$, we have:
\begin{align*}
	\chi_a^{-1}(x) = \overline{\chi_a(x)} = \chi_a(-x), \quad
	\chi_a(x) \chi_a(y) = \chi_a(x + y).
\end{align*}
Hence, for  any $x,y\in\qcube{n}$, we  also have (cf.  \eqref{EQ:krawcharacter}):
\begin{equation*}
\sum_{a \in \qcube{n} : |a| = k} \chi_a(x)\overline{\chi_a(y)} = \sum_{a \in \qcube{n} : |a| = k} \chi_a(x-y) =  \kraw_k(i) \quad \text{when } d(x,y) = |x-y| = i.
\end{equation*}

\subsection*{Invariant kernels} In analogy to the binary case $q=2$, for a degree $r$ univariate polynomial $u\in \R[t]_r$ we define  the associated polynomial kernel  $\kernel(x,y) := u^2(d(x,y))$ ($x,y\in \qcube{n}$) and the associated kernel operator:
\[
	\kernelop : p \mapsto \kernelop p (x) = \int_{\qcube{n}} \overline{p(y)} K(x,y)d\mu(y) = {1\over q^n}\sum_{y \in \qcube{n}} \overline{p(y)}K(x,y) \quad (p \in \redpolys{x}).
\]
Note that $K(x,y)$ is a polynomial on $\qcube{n}$ with degree $2r(q-1)$ in each of $x$ and $y$.  Let us decompose the univariate polynomial $u(t)^2$ in the Krawtchouk basis as $$u(t)^2=  \sum_{i=0}^{2r} \lambda_i \kraw_i(t).$$
Then the  kernel operator $\kernelop$ acts as follows on characters:  for $z \in \qcube{n}$, 
\[
	\kernelop \chi_z = \lambda_{|z|}\chi_z,
\]
which can be seen by retracing the proof of Theorem \ref{THM:FunckHecke}, and we obtain  the Funk-Hecke formula (recall \eqref{EQ:FunkHecke}): for any polynomial $p \in \redpolys{x}_{(q-1)d}$ with Harmonic decomposition
$p=p_0+\ldots+p_d$, 
\[
\kernelop p = \lambda_0 p_0 + \dots + \lambda_d p_d. 
\]

\subsection*{Performing the analysis}
It remains to find a univariate polynomial $u \in \R[t]$ of degree $r$ with $u^2(t) = \sum_{i=0}^{2r}\lambda_i \kraw_i(t)$ for which $\lambda_0 = 1$ and the other scalars $\lambda_i$ are close to $1$. As before (cf. \eqref{eq:krawnormlambda}), we have:
\[
	\lambda_i = \langle \kraw_i, u^2 \rangle_{\omega} / \| \kraw_i \|_{\omega}^2 = \langle \krawnorm_i, u^2 \rangle_{\omega}.
\]
So we would like to minimize $\sum_{i=1}^{2r} (1-\lambda_i)$. We are therefore interested in the inner Lasserre hierarchy applied to the minimization of the function $g(t) = d - \sum_{i = 0}^d \krawnorm_i(t)$ on the set $[0:n]$ (equipped with the measure $\omega$ from (\ref{eqwq})). 
We show first that this function $g$ again has a nice linear upper estimator.

\begin{lemma}
\label{LEM:qkrawabsdist}
We have:
\begin{equation}
\label{EQ:qkrawabsdist}
\begin{aligned}
    |\krawnorm_k(t) - \krawnorm_k(t+1)| &\leq \frac{2k}{n}, \quad\quad &&(t=0,1,\dots,n-1)\\
    |\krawnorm_k(t) - 1| &\le {2k\over n} \cdot t &&(t=0,1,\dots,n)
\end{aligned}
\end{equation}
for all $k \leq n$. 
\end{lemma}

\begin{proof}
The proof is almost identical to that of Lemma \ref{LEM:krawabsdist}. Let $t\in [0:n-1]$ and $0 < k \leq d$. Consider the elements $1^t 0^{n-t}, 1^{t+1} 0^{n-t-1} \in \qcube{n}$ from Lemma \ref{lemxt}. Then we have:
\begin{align*}
|\kraw_k(t) - \kraw_k(t+1)| 
&\overset{\eqref{EQ:qxt}}{=}
 |\sum_{|a| = k} \chi_a(1^t 0^{n-t}) - \chi_a(1^{t+1} 0^{n-t-1})| \\
&\leq 2\cdot \#\big\{ a \in \qcube{n} : |a| = k, a_{t+1} \neq 0 \big\} =2\cdot  (q-1)^k \cdot {n - 1 \choose k-1},
\end{align*}
where for the  inequality we note that $\chi_a(1^{t} 0^{n-t}) = \chi_a(1^{t+1} 0^{n-t-1})$ if $a_{t+1} = 0$. As $\kraw_k(0) = (q-1)^k {n \choose k}$, this implies that:
\[
|\krawnorm_k(t) - \krawnorm_k(t+1)| \leq 2\cdot {n - 1 \choose k-1} / {n \choose k} = \frac{2k}{n}.
\]
This shows the first inequality of \eqref{EQ:qkrawabsdist}. The second inequality follows using the triangle inequality, a telescope summation argument and the fact that $\krawnorm_k(0)=1$. \qedMP
\end{proof}
From Lemma \ref{LEM:qkrawabsdist}  we obtain  that the function $g(t)= d - \sum_{i = 0}^d \krawnorm_i(t) $ admits the following linear upper estimator: $g(t) \leq d(d+1) \cdot (t/n)$ for $t \in [0:n]$.
Now the same arguments as used for the case $q=2$ enable us to conclude:
$$
\upbound{(q-1)r}-\funcmin \leq C_d \cdot \xi_{r+1,q}^n / n
$$
and, when $d(d+1) \xi_{r+1,q}^n / n \leq 1/2$,
\[
	\funcmin - \lowbound{(q-1)r} \leq 2 C_d \cdot \xi_{r+1,q}^n / n.
\]
Here $C_d$ is a constant depending only on $d$ and $\xi_{r+1,q}^n$ is the least root of the Krawtchouk polynomial $\kraw_{r+1, q}$. Note that as the kernel $\kernel(x,y) = u^2(d(x,y))$ is of degree $2(q-1)r$ in $x$ (and $y$), we are only able to analyze the corresponding levels $(q-1)r$ of the hierarchies.
We come back below to the question on how to show the existence of the above constant $C_d$.

But first we finish the analysis. Having shown analogs of Theorem \ref{THM:main} and Theorem \ref{THM:mainup} in this setting, it remains to state the following more general version of Theorem \ref{THM:Levenshtein}, giving information about the smallest roots of the $q$-ary Krawtchouk polynomials.
\begin{theorem}[\cite{Levenshtein}, Section 5]\label{THM:krawrootq}
Fix $t\in [0,{q-1\over q}]$. Then the smallest roots $\xi^n_{r,q}$ of the $q$-ary Krawtchouk polynomials $\kraw_{r, q}$ satisfy:

\begin{equation}\label{eqphitq} 
	\lim_{r/n \to t} \xi_{r,q}^n / n = \varphi_q(t) := \frac{q-1}{q} - \left( \frac{q-2}{q}\cdot t + \frac{2}{q} \sqrt{(q-1)t(1-t)}\right).
\end{equation}
Here the above limit   means that, for any sequences $(n_j)_j$ and $(r_j)_j$ of integers such that $\lim_{j\to\infty} n_j=\infty$ and 
$\lim_{j\to\infty}r_j/n_j=t$, we have $\lim_{j\to\infty} \xi_{r_j,q}^{n_j}/n_j =\varphi_q(t)$.
\end{theorem}

Note that for $q=2$ we have $\varphi_q(t) ={1\over 2} -\sqrt{t(1-t)}$, which is the function $\varphi(t)$ from \eqref{eqphitt}. 
To avoid technical details we only quote in Theorem \ref{THM:krawrootq} the asymptotic analog of Theorem \ref{THM:Levenshtein} (and not the exact bound on the root $\xi^n_{r,q}$ for any $n$).  Therefore we have shown the following  $q$-analog of Corollary \ref{cor2}.

\begin{theorem}\label{theoq}
Fix $d\le n$ and for $n,r\in\N$ write
\begin{align*}
E_{(r)}(n) &:= \sup_{f \in \R[x]_{(q-1)d}} \big\{\funcmin - \lowbound{r} : \|f\|_\infty = 1 \big\}, \\ 
E^{(r)}(n) &:= \sup_{f \in \R[x]_{(q-1)d}} \big\{ \upbound{r} - \funcmin : \|f\|_\infty = 1 \big\}.
\end{align*}
There exists a constant $C_d>0$ (depending also on $q$) such that, for any $t\in [0,{q-1\over q}]$, we have:
$$
\lim_{r/n\to t} E^{((q-1)r)}(n)  \le C_d\cdot \varphi_q(t)
$$
and, if $d(d+1)\cdot \varphi_q(t) \le 1/2$, then we also have:
$$
\lim_{r/n\to t} E_{((q-1)r)}(n)  \le 2\cdot C_d\cdot \varphi_q(t).
$$
Here $\varphi_q(t)$ is the function defined in (\ref{eqphitq}). Recall that  the  limit notation $r/n\to t$  means that the claimed convergence holds  for any sequences $(n_j)_j$ and $(r_j)_j$ of integers such that $\lim_{j\to\infty} n_j=\infty$ and 
$\lim_{j\to\infty}r_j/n_j=t$. 
\end{theorem}
For reference, the function $\varphi_q(t)$ is shown for several values of $q$ in Figure \ref{FIG:qphi}.

\begin{figure}[h!]
    \centering
    \begin{tikzpicture}
        \begin{axis}[ 
        axis lines = middle,
        x label style={at={(1,0)},right},
 		y label style={at={(0,1.03)},above},
 		set layers,                   
        grid, grid style=dashed,
        xlabel=$t$,
        ylabel={$\varphi_q(t)$},
        xmax=1,
        xmin=0,
        ymax = 1,
        width= 0.80 \textwidth,
        height= 0.4 \textwidth,
        xtick = {0.1, 0.2, 0.3, 0.4, 0.5, 0.6, 0.7, 0.8, 0.9, 1.0},
        ytick = {0.1, 0.2, 0.3, 0.4, 0.5, 0.6, 0.7, 0.8, 0.9, 1.0},
        legend cell align={left}
        ] 
        \addplot+[very thick, mark=none, domain=0.000:.5, samples = 200] 
        {(2-1 - (2-2)*x - 2*((2-1)*x*(1-x))^0.5)/2};
        \addplot+[very thick, mark=none, domain=0.000:.666, samples = 200] 
        {(3-1 - (3-2)*x - 2*((3-1)*x*(1-x))^0.5)/3};
        \addplot+[very thick, mark=none, domain=0.000:.75, samples = 200] 
        {(4-1 - (4-2)*x - 2*((4-1)*x*(1-x))^0.5)/4};
        \addplot+[very thick, mark=none, domain=0.000:.8, samples = 200] 
        {(5-1 - (5-2)*x - 2*((5-1)*x*(1-x))^0.5)/5};
        \legend{$q=2$, $q=3$, $q=4$, $q=5$}
        \end{axis}
    \end{tikzpicture}
    \caption{The function $\varphi_q(t)$ for several values of $q$. Note that the case $q=2$ corresponds to the function $\varphi(t)$ of \eqref{eqphitt}.}
    \label{FIG:qphi}
\end{figure}
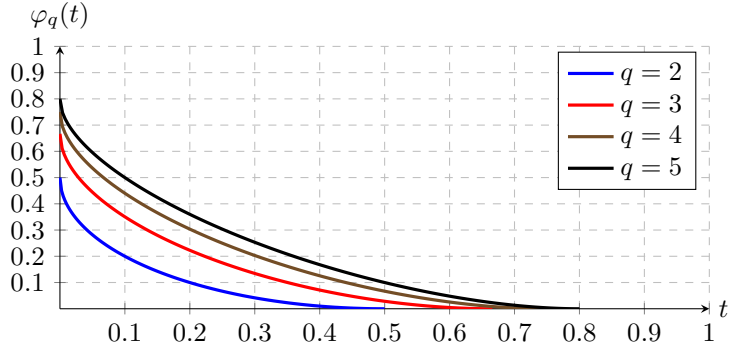

\subsection*{A generalization of Lemma \ref{LEM:cube}}
The arguments above omit a generalization of Lemma \ref{LEM:cube}, which is instrumental  to show the existence of the constant $C_d$ claimed above.
In other words, we still need to show that if $p : \qcube{n} \rightarrow \R$ is a polynomial of degree $(q-1)d$ on $\qcube{n}$ with harmonic decomposition
$p=p_0+\ldots +p_d$, there then exists a constant $\gamma_d > 0$ (independent of  $n$) such that:
\[
	\|p_i\|_{\infty} \leq \gamma_d \| p \|_{\infty} \text{ for all } 0\le i \leq d.
\]
Then, as in the binary case,  we may set $C_d=d(d+1)\gamma_d$.
The proof given in  Section \ref{APP:harmcompbound} for the case $q=2$ applies almost directly to the general case, and we only generalize certain steps as required. So consider again the parameters:
\begin{align*}
    \rho(n, d, k) &:= \sup \{\|p_k\|_\infty : p = p_0 + p_1 + \dots + p_d\in\R[x]_{(q-1)d}, \|p\|_{\infty} \leq 1 \}, \text{ and }\\
    \rho(n, d) &:= \max_{0 \leq k \leq d} \rho(n, d, k).
\end{align*}
Lemmas \ref{LEM:claimsimp1} and \ref{LEM:claimsimp2}, which show that the optimum solution $p$ to $\rho(n, d, k)$ may be assumed to be invariant under $\stab(0) \subseteq \Aut(\qcube{n})$, clearly apply to the case $q > 2$ as well. That is to say, we may assume $p$ is of the form\footnote{Note that as $p$ is assumed to be real-valued, the coefficients $\lambda_i$ must be real. Indeed, for each $a \in \qcube{n}$, we have $\langle p, \chi_a \rangle_{\mu} = \lambda_{|a|} \| \chi_a\|^2 = \lambda_{|a|} \| \chi^{-1}_{a}\|^2 = \langle p, \overline{\chi_a} \rangle_{\mu} = \overline{\langle p, \chi_a \rangle_{\mu}}$.}:
\[	
	p(x) = \sum_{i=0}^d \lambda_i \charsum_i(x) \quad (\lambda_i \in \R)
\]
where $\charsum_i = \sum_{|a| = i}\chi_a \in H_i$ is the zonal spherical function of degree $(q-1)i$ (cf. \eqref{EQ:qzonal} and \eqref{EQ:charsum}).
Using \eqref{EQ:qxt}, we obtain a reformulation of $\rho(n, d, k)$ as an LP (cf. \eqref{EQ:primal}):
\begin{equation}
\label{EQ:qprimal}
    \begin{split}
    \rho(n, d, k) 
    = \max \quad &\lambda_k \\
    s.t. \quad & -1 \leq \sum_{i=0}^d \lambda_i \krawnorm_{i, q}(t) \leq 1 \quad (t=0,1, \dots, n).  
    \end{split}
\end{equation}
For $k \in \N$, let $\krawlim_k(t) := \lim_{n \to \infty} \krawnorm_k(nt) = \Big(1 - \frac{q}{q-1}t\Big)^k$ and consider the program (cf. \eqref{EQ:primallim}):
\begin{equation}
    \label{EQ:qprimallim}
    \begin{split}
    \rho(\infty, d, k) := \max \quad &\lambda_k \\
    s.t. \quad &-1 \leq \sum_{i=0}^d \lambda_i \krawlim_i(t) \leq 1 \quad (t \in [0, 1]).
    \end{split}
\end{equation}
As before, we have $\rho(n, d, k) \leq \rho(\infty, d, k)$, noting that (the proofs of) Lemma \ref{LEM:solextension} and Lemma \ref{LEM:sollim} may be applied directly to the case $q>2$. From there, it suffices to show $\rho(\infty, d, k) < \infty$, which can be argued in an analogous way to the case $q=2$.

\section{Matrix-valued polynomials}
\label{APP:matrixvalued}
In this section, we show how the arguments used for the proofs of our main results in Theorem \ref{THM:main} and Theorem \ref{THM:mainup} may be applied in the setting of matrix-valued polynomials, thereby proving Theorem \ref{THM:mainmatrix} and Theorem \ref{THM:mainupmatrix}. 

Recall that $\matspace$ is the space of $k\times k$ real symmetric matrices  and $\matspace[x] \subseteq \R^{k \times k}[x]=\R[x]^{k\times k}$ is the space of $n$-variate polynomials whose coefficients lie in $\matspace$. Given a polynomial matrix $F\in \matspace[x]$ we consider the matrix-valued polynomial optimization problem:
\begin{equation}\label{eqFmin}
	F_{\min} := \min_{x \in \cube{n}} \lambda_{\min}(F(x)),
\end{equation}
for which we have the outer Lasserre hierarchy:
\begin{equation}
    F_{(r)} := \sup_{\lambda \in \R} \left\{F(x) - \lambda \cdot I = S(x) \text{ on } \cube{n} \text{ for some } S \in \Sigma^{k \times k}_r \right\},
\end{equation}
and the inner Lasserre hierarchy:
\begin{equation}
\label{EQ:APPmatinner}
	 F^{(r)} := \inf_{S \in \Sigma^{k \times k}_r} \left \{ \int_{\cube{n}} {\rm Tr} \big( F(x) S(x)\big) d\mu(x) : \int_{\cube{n}} {\rm Tr} \big(S(x)\big) d\mu(x) = 1\right \}.
\end{equation}
Here, the set $\Sigma^{k \times k}_r$ consists of all sum-of-squares polynomial matrices $S \in \matspace[x]$, of the form:
\[
	S(x) = \sum_{i} U_i(x) U_i(x)^\top \quad (U_i \in \R^{k \times m}[x], ~~{\rm \deg} ~U_i \leq r,~~m\in \N).
\]
\subsection*{The outer hierarchy  (proof of Theorem \ref{THM:mainmatrix})}
We generalize the outline of Section \ref{SEC:overview} to the matrix-valued setting. Let $F\in \matspace[x]$ be the  polynomial matrix of degree $d$ to be optimized, and assume w.l.o.g. that $0 \leq \|F\|_\infty \leq 1$. Here, and throughout this section, $\|F\|_\infty := \max_{x \in \cube{n}} \| F(x) \|$ is the largest absolute value of an eigenvalue of $F(x)$ over $\cube{n}$. A kernel $\kernel$ of the form $\kernel(x, y) = u^2(d(x,y))$ with $u \in \R[t]_r$ (cf. \eqref{EQ:kernelform}) induces a linear operator $\kernelop$ on $\matspace[x]$ by:
\[
	\kernelop P (x) := \int_{\cube{n}} P(y) K(x,y) d \mu(y) = \frac{1}{2^n} \sum_{y \in \cube{n}} P(y) K(x, y) \quad (P \in \matspace[x]).
\]
If $P(x) \succeq 0$ for all $x \in \cube{n}$, then  the polynomial $\kernelop P$ is a sum-of-squares polynomial matrix of degree at most $2r$ on $\cube{n}$. Indeed, we then have:
\[
	\kernelop P(x) =  \frac{1}{2^n} \sum_{y \in \cube{n}} U_y(x) U_y(x)^\top, \quad \text{ where } U_y(x) = u(d(x,y)) \sqrt{P(y)}.
\]

Given $\delta > 0$ to be determined later, set $\tilde F = F + \delta I$. Assuming that $\kernelop$ is non-singular, we can write $F = \kernelop (\kernelop^{-1} \tilde F)$. Therefore, assuming that $\kernelop^{-1} \tilde F$ is positive semidefinite over $\cube{n}$, we find that $F + \delta I$ is a sum-of-squares polynomial matrix of degree $2r$ on $\cube{n}$, and thus that $F_{\min} - F_{(r)} \leq \delta$. 

To guarantee positive semidefiniteness of $\tilde F$, it suffices to ensure that (cf. \eqref{EQ:overview1}):
\[
\|\kernelop^{-1} - I\| := \sup_{P \in \matspace[x]_d} \frac{\| \kernelop^{-1} P  - P \|_\infty}{\|P\|_\infty} \leq \delta.
\]
Indeed, as the smallest eigenvalue of $\tilde F(x)$ is at least $\delta$ for each $x \in \cube{n}$, the smallest eigenvalue of $\kernelop^{-1}F(x)$ must then be at least zero.

As in the case of scalar-valued polynomials, the eigenvalues of $\kernelop$ are given by the coefficients $\lambda_i$ in the expansion ${u^2(t) = \sum_{i = 0}^{2r} \lambda_i \kraw_i(t)}$. Indeed, if $P \in \matspace[x]$ is a  polynomial  matrix of degree $d$ then we may decompose it into harmonic components entry-wise to obtain 
${P(x) = \sum_{i=0}^d P_i(x)}$ and (cf. Theorem \ref{THM:FunckHecke}):
\[
	\kernelop P(x) = \sum_{i=0}^d \lambda_i P_i(x).
\]

It remains to express the quantity $\|\kernelop^{-1} - I\|$ in terms of the eigenvalues $\lambda_i$ of $\kernelop$, after which the proof proceeds as in the case of scalar polynomials. For this, note that:
\[
\|\kernelop^{-1} P - P\|_\infty = \|\sum_{i=1}^d (\lambda_i^{-1} - 1)P_i \|_\infty \leq \sum_{i=1}^d |\lambda_i^{-1} - 1| \|P_i\|_\infty \leq \sum_{i=1}^d |\lambda_i^{-1} - 1| \cdot \gamma_d \| P \|_\infty.
\]
where $\gamma_d$ is the constant of Lemma \ref{LEM:cube} (cf. \eqref{EQ:kernelharmbound}). The last inequality relies on the following generalization of Lemma~\ref{LEM:cube}, whose proof here is essentially as given in \cite{FangFawzi2019}.
\begin{lemma}
\label{LEM:harmboundmatrix}
Let $P(x) = \sum_{i=0}^d P_i(x)$ be a  polynomial matrix of degree $d$, decomposed into harmonic components. If $\gamma_d$ is the constant of Lemma \ref{LEM:cube}, we then have:
\[
	\|P_i\|_{\infty} \leq \gamma_d \|P\|_\infty \quad \text{ for all } i \leq d.
\]
\end{lemma}
\begin{proof}
For any matrix $M \in \matspace$, its spectral norm is  $\| M \|_{} = \max_{y \in \R^k} \{ |y^\top M y| : \|y\| = 1\}$. Therefore, we have:
\[
\|P\|_\infty = \max_{x \in \cube{n}} \max_{\|y\| = 1} |y^\top P(x) y| \quad \text{ and } \quad \|P_i\|_\infty = \max_{x \in \cube{n}} \max_{\|y\| = 1} |y^\top P_i(x) y|.
\]
For fixed $y$, the function $p^y : x \mapsto y^\top P(x) y$ is a (scalar) polynomial on $\cube{n}$ of degree $d$, whose harmonic components are given by $p^y_i : x \mapsto y^\top P_i(x) y$. Therefore, we may invoke Lemma \ref{LEM:cube} to bound:
\[
	\max_{x \in \cube{n}} |p_i^y(x)| \leq \gamma_d \max_{x \in \cube{n}} |p^y(x)| ~ \text{for all } \|y\| = 1,
\]
and conclude that $\|P_i\|_\infty \leq \gamma_d \|P\|_\infty $. \qedMP
\end{proof}

\subsection*{The inner hierarchy (proof of Theorem \ref{THM:mainupmatrix})}
We generalize the arguments of Section \ref{SEC:generalize} to the matrix-valued setting. Let $F$ again be the  polynomial matrix of degree $d$ to be optimized, and assume w.l.o.g. that $0 \leq \|F\|_\infty \leq 1$ and that the minimum in the optimization problem (\ref{eqFmin})  is attained at $0$, i.e., that $F_{\min} = \lambda_{\min} \big(F(0)\big)$.

As in the scalar case, we work to reduce problem \eqref{EQ:APPmatinner} to a (now matrix-valued) instance of the inner hierachy in one variable. Note first that $F^{(r)} \leq F_{\rm sym}^{(r)}$ for each $r \in \N$, where $F_{\rm sym}^{(r)}$ is obtained by restricting the optimization in \eqref{EQ:APPmatinner} to  polynomial matrices  $S(x)$ of the form $S(x) = U(|x|)$. Writing $\widehat F$ for the univariate  polynomial matrix satisfying 
\[
	\widehat F(|x|) = \frac{1}{|\stab(0)|}\sum_{\sigma \in \stab(0)} F \circ \sigma (x) \quad (x \in \cube{n}),
\]
we find (cf. \eqref{EQ:symmetrizedinner}):
\begin{equation}
\label{EQ:symmetrizedinnermatrix}
	F_{\rm sym}^{(r)} = \min_{U \in \Sigma^{k \times k}_r[t]} \Big\{ \int_{[0:n]} {\rm Tr}\big(\widehat F(t)U(t)\big) d\omega(t) : \int _{[0:n]} {\rm Tr}\big(U(t)\big) d \omega(t) = 1 \Big\}.
\end{equation}
It remains to analyze the program \eqref{EQ:symmetrizedinnermatrix}. We first give a linear upper estimator for $\widehat F$ (cf. Lemma \ref{LEM:Flinearupperestimator}).
\begin{lemma}
For all $t \in [0:n]$, we have:
\[
\widehat F(t) \preceq \widehat G(t) := d(d+1) \cdot \gamma_d \cdot t/n \cdot I + C,  
\]
where $C=\widehat F(0)$ is a constant matrix with $\lambda_{\min}(C) = 0$.
\end{lemma}
\begin{proof}
We may write $\widehat F(t) = \sum_{i=0}^d \Lambda_i \krawnorm_i(t)$ for certain $\Lambda_i \in \matspace$. We then have:
\begin{align*}
\sum_{i=0}^d \Lambda_i \krawnorm_i(t) &= \sum_{i=0}^d \Lambda_i \big(\krawnorm_i(t) - 1 \big) + \sum_{i=0}^d \Lambda_i \\
& \preceq \max_{i=0}^d \|\Lambda_i\|_\infty \cdot I \cdot \sum_{i=0}^d |1 - \krawnorm_i(t)| + \sum_{i=0}^d \Lambda_i \\
& \preceq d(d+1) \cdot \gamma_d \cdot t/n \cdot I + \sum_{i=0}^d \Lambda_i,
\end{align*}
making use of Lemma \ref{LEM:harmboundmatrix} and \eqref{EQ:krawabsdist} for the final inequality. It remains to note that $\sum_{i=0}^d \Lambda_i = \widehat F(0)$, and that $\lambda_{\min} (\widehat F(0)) = 0$ by assumption. \qedMP
\end{proof}
As $\widehat F(t) \preceq \widehat G(t)$ for all $t \in [0:n]$, we have $F_{\rm sym}^{(r)} \leq \widehat G^{(r)}_{[0:n], \omega}$, where:
\[
	\widehat G^{(r)}_{[0:n], \omega} = \min_{U \in \Sigma^{k \times k}_r[t]} \Big\{ \int_{[0:n]} {\rm Tr}\big(\widehat G(t)U(t)\big) d\omega(t) : \int _{[0:n]} {\rm Tr}\big(U(t)\big) d \omega(t) = 1 \Big\}
\]
is the inner Lasserre hierarchy for $G$ computed on $[0:n]$ w.r.t. the measure $\omega$.
To conclude the argument, we prove the following generalization of Theorem \ref{THM:deKlerkLaurentinner} (see also Remark \ref{REM:linearunivariate}).

\begin{corollary}
Let $G(t) = ct \cdot I + C$ be a linear matrix-valued polynomial with $c>0$ and $\lambda_{\min}(C) = 0$. Then we have:
\[
G^{(r)}_{[0:n], \omega} \leq c \cdot \xi^n_{r+1},
\] 
where $\xi^n_{r+1}$ is the least root of the degree $r+1$ Krawtchouk polynomial.

\begin{proof}
Let $u$ be a unit eigenvector for $C$ corresponding to (one of its) zero eigenvalues. Then for any univariate sum-of-squares polynomial $s \in \Sigma_r$, the matrix-valued polynomial $U(t) = s(t) u u^\top$ is a sum-of-squares polynomial matrix of degree $2r$. Furthermore, for such a $U$ we have:
\[
	\int_{[0:n]} {\rm Tr}\big( G(t)U(t)\big) d\omega(t) = \int_{[0:n]} ct \cdot s(t) d\omega(t)
\]
and
\[
	\int_{[0:n]} {\rm Tr}\big(U(t)\big) d\omega(t) = \int_{[0:n]} s(t) d\omega(t).
\]
Therefore, writing $g(t) = ct$, and making use of Theorem \ref{THM:deKlerkLaurentinner} and Remark \ref{REM:linearunivariate}, we have:
\[
G^{(r)}_{[0:n], \omega} \leq \inf_{s \in \Sigma_r} \left\{\int_{[0:n]} ct \cdot s(t) d\omega(t) : \int_{[0:n]} s(t) d\omega(t) = 1 \right\} = g^{(r)}_{[0:n], \omega} = c \cdot \xi_{r+1}^n.
\]
This concludes the proof. \qedMP
\end{proof}
\end{corollary}

%
%


\begin{thebibliography}{10}

\bibitem{AN2004}
N. Alon and A. Naor.
Approximating the cut-norm via Grothendieck's inequality.
{36th annual ACM Symposium on Theory of Computing}, pp. 72--80, 2004.

\bibitem{ABHKS}
S. Arora, E. Berger, E. Hazan, G. Kindler and M. Safra.
On non-approximability for quadratic programs.
{\em Proceedings of the 46th Annual IEEE Symposium on Foundations of Computer Science}, pp. 206--215, 2005.

\bibitem{BCC1993}
E. Balas, S. Ceria and G. Cornu\'ejols.
A lift-and-project cutting plane algorithm for mixed 0--1 programs. 
{\em Mathematical Programming} \textbf{58}:295--324, 1993.

\bibitem{BBC}
N. Bansal, A. Blum and S. Chawla.
Correlation clustering.
{\em Machine Learning}, \textbf{46}(1-3):89--113, 2004.

\bibitem{Barak}
B. Barak and D. Steurer.
Sum-of-squares proofs and the quest toward optimal algorithms. In {\em Proceedings of International Congress of Mathematicians (ICM)}, 2014.

\bibitem{CW2004}
M. Charikar and A. Wirth.
Maximizing quadratic programs: Extending Grothendieck's inequality. 
{\em Proceedings of the 45th Annual IEEE Symposium on Foundations of Computer Science}, pp. 54--60, 2004.

\bibitem{deKlerkLaurentSurvey}
E.~de~Klerk and M.~Laurent.
\newblock {\em A survey of semidefinite programming approaches to the
  generalized problem of moments and their error analysis}. In: Araujo C., Benkart G., Praeger C., Tanbay B. (eds), World Women in Mathematics 2018. Association for Women in Mathematics Series, vol 20. Springer, Cham, pp. 17--56, 2019.

\bibitem{deKlerkLaurent2019}
E.~de~Klerk and M.~Laurent.
\newblock {Convergence analysis of a {L}asserre hierarchy of upper bounds for
  polynomial minimization on the sphere}.
\newblock {\em Mathematical Programming}, 2020.
\url{https://doi.org/10.1007/s10107-019-01465-1}

\bibitem{deKlerkLaurent2018}
E.~de~Klerk and M.~Laurent.
Worst-case examples for Lasserre's measure based hierarchy for polynomial optimization on the hypercube.
\newblock {\em Mathematics of Operations Research}, \textbf{45}(1):86--98, 2020.

\bibitem{Doherty_Wehner2012}
A.~C. {Doherty} and S.~{Wehner}.
\newblock {Convergence of SDP hierarchies for polynomial optimization on the
  hypersphere}.
\newblock {\em arXiv preprint}, 2012. \url{arXiv:1210.5048}.

\bibitem{FangFawzi2019}
K.~Fang and H.~Fawzi.
\newblock The sum-of-squares hierarchy on the sphere, and applications in
  quantum information theory.
\newblock {\em Mathematical Programming}, 2020. 
\url{https://doi.org/10.1007/s10107-020-01537-7}

\bibitem{FawziSaundersonParillo2016}
H.~Fawzi, J.~Saunderson, and P.~A. Parrilo.
\newblock Sparse sums of squares on finite abelian groups and improved
  semidefinite lifts.
\newblock {\em Mathematical Programming}, \textbf{160}(1-2):149-191, 2016.

\bibitem{GW}
M. Goemans and D. Williamson.
Improved approximation algorithms for maximum cut and satisfiability problems using semidefinite programming.
{\em Journal of the Association for Computing Machinery}, \textbf{42}(6):1115--1145, 1995.

\bibitem{KarlinMathieuNguyen2011}
A.~R. Karlin, C.~Mathieu, and C.~T.~Nguyen.
\newblock Integrality gaps of linear and semi-definite programming relaxations
  for knapsack.
\newblock O. G{\"u}nl{\"u}k and Gerhard~J. Woeginger (eds.) {\em
  Integer Programming and Combinatoral Optimization}, pp. 301--314, Springer Berlin,
  Heidelberg, 2011.



\bibitem{KurpiszLeppanenMastrolilli} 
A. Kurpisz, S. Leppänen, M. Mastrolilli.
\newblock Tight Sum-of-Squares lower bounds for binary polynomial optimization problems. \newblock I. Chatzigiannakis et al. (eds.) {\em 43rd International Colloquium on Automata, Languages, and Programming (ICALP 2016)}, \textbf{78}:1--14, 2016.

\bibitem{Lasserre2001}
J.B. Lasserre.
\newblock Global optimization with polynomials and the problem of moments.
\newblock {\em SIAM Journal on Optimization}, \textbf{11}(3):796--817, 2001.

\bibitem{LasserreORL2016}
J.B. Lasserre.
\newblock A max-cut formulation of $0/1$ programs.
\newblock {\em Operations Research Letters}, \textbf{44}:158--164, 2016.

\bibitem{Lasserre2001b}
J.B. Lasserre.
\newblock An explicit exact sdp relaxation for nonlinear 0-1 programs.
\newblock K. Aardal and B. Gerards (eds.) {\em Integer Programming
  and Combinatorial Optimization}, pp. 293--303, Springer Berlin, Heidelberg, 2001.

\bibitem{Lasserre2009}
J.B. Lasserre.
{\em  Moments, Positive Polynomials and Their Applications.} Imperial College Press, London
(2009)


\bibitem{Lasserre2010}
J.B. Lasserre.
\newblock A new look at nonnegativity on closed sets and polynomial
  optimization.
\newblock {\em SIAM Journal on Optimization}, \textbf{21}(3):864--885, 2010.

\bibitem{Laurent2001}
M.~Laurent.
\newblock A comparison of the Sherali-Adams, Lov\'asz-Schrijver and Lasserre
  relaxations for 0-1 programming.
\newblock {\em Mathematics of Operations Research}, \textbf{28}(3), pp. 470--496 (2003)

\bibitem{Laurent2003}
M.~Laurent.
\newblock Lower bound for the number of iterations in semidefinite hierarchies
  for the cut polytope.
\newblock {\em Mathematics of Operations Research}, \textbf{28}(4):871--883, 2003.

\bibitem{Laurentfinite}
M.~Laurent.
\newblock Semidefinite representations for finite varieties.
\newblock {\em Mathematical Programming}, \textbf{109}:1--26, 2007.

\bibitem{Laurent2009}
M. Laurent. 
Sums of squares, moment matrices and optimization over polynomials. In {\em Emerging Applications of Algebraic Geometry,} Vol. 149 of IMA Volumes in Mathematics and its Applications, M. Putinar and S. Sullivant (eds.), Springer, pp. 157-270, 2009.

\bibitem{SlotLaurent2020}
L.~Slot and M.~Laurent.
\newblock {Near-optimal analysis of of Lasserre's univariate measure-based
bounds for multivariate polynomial optimization}.
{\em Mathematical Programming}, 2020.
\url{https://doi.org/10.1007/s10107-020-01586-y}

\bibitem{Lee}
J.R.  Lee, P. Raghavendra and D. Steurer.
Lower Bounds on the Size of Semidefinite Programming Relaxations.
In STOC '15: Proceedings of the forty-seventh annual ACM symposium on Theory of Computing, pp. 567--576, 2015.


\bibitem{Levenshtein}
V.~I. {Levenshtein}.
\newblock Universal  bounds  for  codes  and  designs.
\newblock {\em Handbook  of  Coding Theory}, vol. 9,  pp. 499--648, North-Holland, Amsterdam,  1998.

\bibitem{LovaszSchrijver1991}
L. Lov\'asz and A. Schrijver.
Cones of matrices and set-functions and 0--1 optimization. 
{\em SIAM Journal on Optimization}, \textbf{1}:166--190, 1991.

\bibitem{MacwilliamsSloane1983}
F.~Macwilliams and N.~Sloane.
\newblock {The Theory of Error Correcting Codes}, vol. 16 of {\em
  North-Holland Mathematical Library},
\newblock Elsevier, 1983.

\bibitem{Natanson1964}
I.P. Natanson.
Constructive Function Theory, Vol. I Uniform Approximation, 1964

\bibitem{NieSchweighofer2007}
J.~Nie and M.~Schweighofer.
\newblock On the complexity of Putinar's positivstellensatz.
\newblock {\em Journal of Complexity}, \textbf{23}(1):135--150, 2007.

\bibitem{ODonnell2017}
R. O'Donnell.
\newblock SOS is not obviously automatizable, even approximately.
\newblock {\em 8th Innovations in Theoretical Computer Science Conference},
\textbf{59}:1--10, 2017.

\bibitem{Parrilo2000}
P.~A.~Parrilo.
\newblock Structured semidefinite programs and semialgebraic geometry methods
  in robustness and optimization, 2000.
\newblock PhD thesis, California Institure of Technology.

\bibitem{Raghavendra2017}
P. Raghavendra and B. Weitz.
\newblock On the bit complexity of sum-of-squares proofs.
\newblock {\em 44th International Colloquium on Automata, Languages, and Programming}, \textbf{80}:1--13,
 2017.

\bibitem{Reznick1995}
B.~Reznick.
\newblock Uniform denominators in Hilbert's seventeenth problem.
\newblock {\em Mathematische Zeitschrift}, \textbf{220}(1):75--97, 1995.

\bibitem{Rothvoss}
T. Rothvoss. The Lasserre hierarchy in approximation algorithms. 
Lecture Notes for the MAPSP 2013 Tutorial, 2013.


\bibitem{SakaueTakedaKim2017}
S.~Sakaue, A.~Takeda, S.~Kim, and N.~Ito.
\newblock Exact semidefinite programming relaxations with truncated moment
  matrix for binary polynomial optimization problems.
\newblock {\em SIAM Journal on Optimization}, \textbf{27}(1):565--582, 2017.

\bibitem{SchererHol}
C.W. Scherer and C.W.J. Hol.
Matrix sum-of-squares relaxations for robust semi-definite programs.
{\em Mathematical Programming}, \textbf{107}:189--211, 2006.

\bibitem{Schweighofer2004}
M.~Schweighofer.
\newblock On the complexity of Schmüdgen's positivstellensatz.
\newblock {\em Journal of Complexity}, \textbf{20}(4):529--543, 2004.


\bibitem{SheraliAdams1990}
H.D. Sherali and W.P. Adams.
A hierarchy of relaxations between the continuous and convex hull representations for zero-one programming problems. 
{\em SIAM Journal on Discrete Mathematics}, \textbf{3}:411--430, 1990.


\bibitem{SlotLaurent2019}
L.~Slot and M.~Laurent.
\newblock Improved convergence analysis of Lasserre's measure-based upper
  bounds for polynomial minimization on compact sets.
\newblock {\em Mathematical Programming}, 2020. 
\url{https://doi.org/10.1007/s10107-020-01468-3}

\bibitem{Szego1959}
G.~Szeg{\"o}.
\newblock {Orthogonal Polynomials}.
\newblock vol. 23 in {\em American Mathematical Society colloquium
  publications}. American Mathematical Society, 1959.

\bibitem{Terras1999}
A.~Terras.
\newblock { Fourier Analysis on Finite Groups and Applications}.
\newblock {\em London Mathematical Society Student Texts}. Cambridge University
  Press, 1999.

\bibitem{Tuncel2010}
L. Tun\c{c}el.
{Polyhedral and Semidefinite Programming Methods in Combinatorial Optimization}.
{\em Fields Institute Monograph}, 2010.

\bibitem{VallentinLN}
F. Vallentin.
\newblock {Semidefinite programs and harmonic analysis}, 2008.
\url{https://arxiv.org/abs/0809.2017}.


\end{thebibliography}
\end{document}